\documentclass[12 pt]{amsart}
  \usepackage[utf8]{inputenc}
  \usepackage[margin=3cm]{geometry}
\makeatletter
\def\l@subsection{\@tocline{1}{0pt}{2pc}{1pc}{}}
\def\l@subsubsection{\@tocline{2}{0pt}{2pc}{1pc}{}}
\makeatother

  \usepackage{stmaryrd}
\usepackage{amsmath,amssymb,amsthm, mathrsfs, mathtools, bbold, mathbbol}
\usepackage{amscd,mathtools}
\usepackage{comment}
  \usepackage{geometry}
  \usepackage{graphicx,epstopdf}
  \usepackage{color,xcolor}
  \usepackage{enumerate}
  \usepackage{xspace}
  \usepackage{tipa}
  \usepackage[user,titleref]{zref}
 \usepackage{soul}
 \usepackage{xcolor}
  \usepackage{epsfig}

\usepackage{tikz}
\usetikzlibrary{shapes.geometric}
\usetikzlibrary{angles}

\usepackage{caption, subcaption}
\usepackage{tikz-cd}
\usetikzlibrary{calc,decorations.pathreplacing}
\usepackage{tikz}
\usetikzlibrary{calc}
\usetikzlibrary{arrows}
\usetikzlibrary{decorations.pathreplacing}
\usetikzlibrary{intersections}
\usepackage{xcolor}
\usepackage{xspace}
\usepackage{todonotes}

\usepackage{mathrsfs}
\usetikzlibrary{matrix}
\usetikzlibrary{positioning}
\DeclareMathAlphabet{\pazocal}{OMS}{zplm}{m}{n}
\tikzset{>=stealth}

  \usepackage{hyperref}

  \hypersetup{
    colorlinks=true,
    linkcolor=blue,
    filecolor=magenta,      
    urlcolor=cyan,
    citecolor=purple,
    pdftitle={Overleaf Example},
    pdfpagemode=FullScreen,
    }

\usepackage{float}

  \newcommand{\calA}{\mathcal{A}}

  \newcommand{\calN}{\mathcal{N}}

  \newcommand{\calU}{\mathcal{U}}
    \newcommand{\calV}{\mathcal{V}}

  \newcommand{\NN}{\mathbb{N}}

  \newcommand{\RR}{\mathbb{R}}

  \newcommand{\ZZ}{\mathbb{Z}}

  \newcommand{\bfa}{\textbf{a}}
  \newcommand{\bfb}{\textbf{b}}

  \newcommand{\gothic}{\mathfrak}

  \newcommand{\go}{{\gothic o}}

  \newtheorem{theorem}{Theorem}[section]
  \newtheorem{proposition}[theorem]{Proposition}
  \newtheorem{corollary}[theorem]{Corollary}
  \newtheorem{lemma}[theorem]{Lemma}

  \newtheorem{question}[theorem]{Question}
  \newtheorem{claim}[theorem]{Claim}
  
  \newtheorem*{claim*}{Claim}

  \newtheorem{introthm}{Theorem}

  \theoremstyle{definition}
  \newtheorem{definition}[theorem]{Definition}

  \newtheorem*{question*}{Question}
  \newtheorem*{answer*}{Answer}
  \newtheorem*{application*}{Application}

  \theoremstyle{remark}
  \newtheorem{remark}[theorem]{Remark}
  \newtheorem*{remark*}{Remark}
  
  \newcommand{\thmref}[1]{Theorem~\ref{#1}}
  \newcommand{\corref}[1]{Corollary~\ref{#1}}

  \newcommand{\pka}{\partial_{\kappa}}

  \newcommand{\sC}{{\sf C}}   
  \newcommand{\sD}{{\sf D}}

 \newcommand{\sK}{{\sf K}}

  \newcommand{\sQ}{{\sf Q}}   
  \newcommand{\sR}{{\sf R}}

  \renewcommand{\bfa}{{\sf a}}

  \newcommand{\kk}{{\sf k}}   
       
  \newcommand{\mm}{{\sf m}}   
  \newcommand{\nn}{{\sf n}}   
     
  \newcommand{\pp}{{\sf p}}       
  \newcommand{\qq}{{\sf q}}   
  \newcommand{\rr}{{\sf r}}

  \newcommand{\yy}{{\sf y}}

    \newcommand{\ps}{{\partial_s}}

\DeclareMathOperator{\diam}{diam}

\newcommand{\CAT}{\ensuremath{\operatorname{CAT}(0)}\xspace}      

  \newcommand{\param}{{\mathchoice{\mkern1mu\mbox{\raise2.2pt\hbox{$
  \centerdot$}}
  \mkern1mu}{\mkern1mu\mbox{\raise2.2pt\hbox{$\centerdot$}}\mkern1mu}{
  \mkern1.5mu\centerdot\mkern1.5mu}{\mkern1.5mu\centerdot\mkern1.5mu}}}

\DeclarePairedDelimiterX{\norm}[1]{\lvert}{\rvert}{#1}
\DeclarePairedDelimiterX{\Norm}[1]{\lVert}{\rVert}{#1}

  \newcommand{\ST}{\mathbin{\Big|}} 
  \newcommand{\from}{\colon\thinspace} 
  \newcommand{\ep}{\epsilon}

\newcommand{\p}{\partial}

\title[First passage percolation preserves sublinearly Morse boundaries]{First passage percolation preserves sublinearly Morse boundaries}
\author{Sagnik Jana}
 \address{Department of Mathematics,  University of Tennessee at Knoxville, Knoxville, TN, USA}
\email{sjana1@vols.utk.edu}

  \author{Yulan Qing}
 \address{Department of Mathematics,  University of Tennessee at Knoxville, Knoxville, TN, USA}
 \email{yqing@utk.edu}

\begin{document}
\maketitle
\begin{abstract}
Sublinearly Morse directions in proper geodesic spaces are defined by sublinearly Morse stability. In this paper we offer an alternative  characterization for sublinearly Morse geodesic lines via middle recurrence. We then study first passage percolation (FPP) on proper geodesic graphs of bounded degree. We associate an i.i.d. collection of random passage times to each edge. Under suitable conditions on the passage time distribution,  we prove that sublinearly Morse boundaries are invariant under first passage percolation.
\end{abstract}

\section{Introduction}
 A sublinearly Morse boundary \cite{QRT1,QRT2} is a topological space consists of quasi-geodesic rays in a given metric space. It is constructed to extend the Gromov boundary beyond the hyperbolic setting. If the metric space in question, denoted $X$, is proper and geodesic, then the sublinearly Morse boundary, denoted $\ps X$,  is always quasi-isometrically invariant and metrizable. In many situations $\ps X$ provides a natural group-invariant topological model compatible with probabilistic boundary theory for random walks on $X$. A closely related topological set, the Morse boundary \cite{Cordes,CS} of metric spaces is often non-metrizable and does not always realize the Poisson boundary. These distinctions make the sublinearly Morse boundary particularly well-suited for studying asymptotic geometry alongside stochastic and dynamical phenomena. It is natural to ask how sublinearly Morse directions behave under random perturbations of the geometry. One natural model for such perturbations is first passage percolation (FPP), which systematically introduce randomness to graphs. In this process, one assigns i.i.d.\ random weights to the edges, and then defines the distance between two vertices as the smallest total weight among all paths joining them. The path that achieves this minimum is called a geodesic (an optimal path)[see section \ref{fpp}]. For the Cayley graph of $\mathbb{Z}^2$ with respect to the standard generators, this model was introduced by Hammersley and Welsh \cite{HW65}. While first passage percolations on $\mathbb{Z}^2$ and, more generally, on $\mathbb{Z}^d$ has been studied extensively and proven to be hard. For instance, it is still unknown whether bi-infinite geodesics exist (see \cite{Kesten}) after first passage percolation on $\mathbb{Z}^2$. Much less is known for other metric spaces until recently. Existence of bi-infinite geodesics and other interesting results \cite{BZ12,BT17} were obtained when the graph is Gromov hyperbolic.
 
Indeed, \cite{BT17} shows that if $X$ contains a bi-infinite Morse geodesic line, then $X_\omega$ contains a bi-infinite geodesic line almost surely. However, one surprise is that  \cite{BJQ} demonstrates that FPP does not, in general, preserve the Morseness of geodesic lines (a property all geodesic lines in hyperbolic spaces enjoy). This motivates the authors of this paper to study the problem in a setting that contains more hyperbolic behavior than $\ZZ^d$ but less than that of a hyperbolic graph. Specifically, we are interested in understand what happends to \emph{sublinear Morse} directions: recently \cite{JanaQing1} establishes that if $X$ has a bi-infinite sublinearly Morse geodesic line, then $X_\omega$ contains a bi-infinite geodesic line almost surely. Therefore, the immediate motivating question for this paper is:
\begin{question}
Are sublinearly Morseness of sets almost surely preserved under first  passage percolations? 
\end{question}

We answer the question positively:
\begin{introthm}\label{introthm1}
Let $X$ be a graph with uniformly bounded degree and $\p_s X \neq \emptyset$.  Assume that the edge length distribution has finite expectation and also $\nu(\{0\}) = 0$. Then let $\omega$ be induced graphs with weight functions. Then for almost every $\omega$, such the first passage percolation induces a homeomorphism on the sublinearly Morse boundaries: $\p_s X \simeq \p_s X_\omega$.
\end{introthm}

This result opens up the possibility of analysis into the interaction between sublinearly Morse sets and first passage percolations. We give several natural and immediate questions that can be investigated.

Since first passage percolation does not preserve quasi-geodesics, this paper depends on characterizations of sublinear Morseness using paths of bounded slope. This notion is established for Morse geodesic rays and is utilized  in \cite{DMS} and improved in \cite{ADT}. We first develop an proper analogue of such a notion in the context of sublinear Morseness:
\begin{definition} \label{def:middle recurrent}
We say that a path (geodesic line or quasi-geodesic) $\gamma$ is \emph{middle third recurrent} if for every $C\geq1$ there is a constant $c$ (depends on $C$ and $\gamma$) such that the following holds: 
    for any path $p$ with endpoints $a,b \in \gamma$ satisfying $\ell(p)\leq Cd(a,b)$, intersects the $c\cdot \kappa$-neighborhood of the middle third of $\gamma[a,b]$ for some sublinear function $\kappa$. That is,
\[ p \cap \calN_{\kappa} (\gamma_{\frac 13[a,b]}, c )\neq \emptyset.\]
Where, $\gamma_{\frac{1}{3}[a,b]} := \biggl\{x\in \gamma : \min \{d(x,a),d(x,b)\} \ge \frac{1}{3}\cdot d(a,b)\biggr\}$.
\end{definition}
We are able to establish the path characterization of sublinearly Morse sets:

\begin{introthm}\label{introthm2}
A geodesic ray is sublinearly Morse if and only if it has the middle third recurrent property.

\end{introthm}

We remark that it is communicated to the authors during the preparation of this document that one of the directions of this characterization, Proposition 3.4, is also in the process of being established by Raghunath
Jaganathan Vennila as part of his research project.

\subsection*{History in the boundary construction}
The sublinearly Morse boundary is a geometric boundary construction generalizing the Gromov boundary. It is always quasi-isometrically invariant, metrizable, and often provides a group-invariant topological model for suitable random walks. By contrast, the Morse boundary, constructed in \cite{Cordes,CS}, frequently fails to be metrizable and does not always realize the Poisson boundary.

Recent work also highlights strong genericity properties for sublinearly Morse directions. In \cite{GQR22}, genericity with respect to the Patterson–Sullivan measure was established for actions admitting a strongly contracting element. The stationary measure results of \cite{GQR22} were later claimed in a different context by Choi \cite{Choi} using a pivoting method of Gou\"ezel \cite{Gou22}. Following \cite{wenyuan}, genericity on the horofunction boundary was proved for all proper statistically convex-cocompact actions \cite{QY24}. 

\subsection*{Further investigations}
There are several immediate directions of interest as a result of this study. We record them here and offer also existing related results that may be useful in pursuing these questions.
\begin{enumerate}
\item \textit{Are hyperbolic graphs mapped to hyperbolic graphs by first passage percolation if the support is finite?}

 A recent preprint by Dominic Bair and the authors of this paper \cite{BJQ} shows that, assuming the passage time distribution have infinite support, then the hyperbolicity of the graphs are almost surely not preserved. We ask further if the infinite support assumption can be removed or weakened.
\item \textit{
Does first passage percolation induced a H\" older map on the associated sublinearly Morse boundaries?}

Theorem~\ref{introthm1} establishes that a first passage percolation of suitable assumption induced a homeomorphism between the sublinearly Morse boundaries. Given that sublinearly Morse boundaries are metrizable topological spaces, we ask if stronger observations can be made. It is established in \cite{DV15} that radial maps induced H\"older map on Gromov boundaries and \cite{BJQ} showed that first passage percolation exhibits certain radial property. H\"older equivalence is a relation on metric spaces and sublinearly Morse boundaries are known to be metrizable via topological arguments. In an upcoming preprint by Manisha Garg and the authors of this paper, we construct a visual metric for the sublinearly Morse boundary. Thus this question can be now investigated under this specific metric.

\item \textit{Do FPP velocity exist along sublinearly Morse directions?} Even though the image of Morse geodesic rays are not Morse under FPP, in \cite{BM1, BM2}, the authors still studied a great deal of probabilistic properties of the Morse geodesics after FPP. They established that velocity and coalescence exists, among other results. 
\end{enumerate}
\subsection{Acknowledgement} The second-named author would like to thank Gabriel Pallier and Romain Tessera for helpful discussions. The second-named author is supported by Simons Foundation grant [SFI-MPS-TSM-00014066, Y.Q.]

\section{preliminaries}

\subsection{Quasi-isometry and quasi-isometric embeddings}

\begin{definition}[Quasi Isometric embedding] \label{Def:Quasi-Isometry} 
Let $(X , d_X)$ and $(Y , d_Y)$ be metric spaces. For constants $\kk \geq 1$ and
$\sK \geq 0$, we say a map $f \from X \to Y$ is a 
$(\kk, \sK)$--\textit{quasi-isometric embedding} if, for all points $x_1, x_2 \in X$
$$
\frac{1}{\kk} d_X (x_1, x_2) - \sK  \leq d_Y \big(f (x_1), f (x_2)\big) 
   \leq \kk \, d_X (x_1, x_2) + \sK.
$$
If, in addition, every point in $Y$ lies in the $\sK$--neighbourhood of the image of 
$f$, then $f$ is called a $(\kk, \sK)$--quasi-isometry. When such a map exists, $X$ 
and $Y$ are said to be \textit{quasi-isometric}. 
\end{definition}

A \emph{geodesic ray} in $X$ is an isometric embedding $\beta \from [0, \infty) \to X$. We fix a basepoint $\go \in X$ and always assume 
that $\beta(0) = \go$, that is, a geodesic ray is always assumed to start from 
this fixed basepoint. 
\begin{definition}[Quasi-geodesics] \label{Def:Quadi-Geodesic} 
In this paper, a \emph{quasi-geodesic ray} is a continuous quasi-isometric 
embedding $\beta \from [0, \infty) \to X$  starting from the basepoint $\go$. 
\end{definition}
The additional assumption that quasi-geodesics are continuous is not necessary for the results in this paper to hold,
but it is added for convenience and to make the exposition simpler. 

If $\beta \from [0,\infty) \to X$ is a $(\qq, \sQ)$--quasi-isometric embedding, 
and $f \from X \to Y$ is a $(\kk, \sK)$--quasi-isometry then the composition 
$f \circ \beta \from [t_{1}, t_{2}] \to Y$ is a quasi-isometric embedding, but it may 
not be continuous. However, one can adjust the map slightly to make it continuous 
(see Definition 2.2  \cite{QRT1}) such that $f \circ \beta$ is a $(\kk\qq, 2(\kk\qq + \kk \sQ + \sK))$--quasi-geodesic ray.  

Similar to above, a \emph{geodesic segment} is an isometric embedding 
$\alpha \from [t_{1}, t_{2}] \to X$ and a \emph{quasi-geodesic segment} is a continuous 
quasi-isometric embedding \[\alpha \from [t_{1}, t_{2}] \to X.\] 

\noindent \textbf{Notation}. 
Let $\beta$ be a quasi-geodesic ray. Define 
\[
\Norm{x} : = d(\go, x).
\]
For $\rr>0$, let $t_\rr$ be the first time where $\Norm{\beta(t)}=\rr$ and define:
\begin{equation}\label{notation}
\beta_\rr := \beta(t_\rr)
\qquad\text{and}\qquad
\beta|_{\rr} : = \beta{[0,t_\rr]} 
\end{equation}
which are points and segments in $X$, respectively. 

\subsubsection{Sublinearly Morse geodesic lines} \hfill

Let $\kappa \from [0, \infty) \to [1, \infty)$ be a sublinear function that is monotone increasing and concave. That is
\[
\lim_{t \to \infty} \frac{\kappa(t)}{t} = 0. \label{subfunction}
\]

The assumption that $\kappa$ is increasing and concave makes certain arguments
cleaner; otherwise, they are not really needed. One can always replace any 
sub-linear function $\kappa$, with another sub-linear function $\overline \kappa$
so that \[\kappa(t) \leq \overline \kappa(t) \leq \sC \, \kappa(t)\] for some constant $\sC$ 
 and $\overline \kappa$ is monotone increasing and concave.

\begin{lemma} \label{lemma2.3}
Suppose $\kappa$ is any sublinear function. For any $D>0$, there exists $D_1$, $D_2 > 0$ depending on $D$ and $\kappa$ such that, for $x, y \in X,$
if $d(x, y) \leq D\cdot max \{\kappa(\Norm x),\kappa(\Norm y)\}$ then $ D_1\cdot \kappa(\Norm{x}) \leq \kappa (\Norm{y})\leq D_2\cdot \kappa (\Norm x).$
\end{lemma}
\begin{proof}
    The proof is identical to the proof of Lemma 3.2 \cite{QRT1}. Since $\kappa$ is sublinear, there exists a constant $R\geq0$ such that for every $t>0$, $\kappa(t) \leq \frac{t}{2D}+R$.
    If $\Norm x \leq \Norm y$ then,
    \[\Norm y- \Norm x \leq d(x,y)\leq  D \cdot \max \{\kappa(\Norm x),\kappa(\Norm y)\}  \leq \frac{1}{2} (\Norm x+ \Norm y)+2RD.\]
    Hence, $\Norm y \le 3 \Norm x+2RD.$ \newline \newline
    Therefore, \[\kappa(\Norm x) \le \kappa(\Norm y)\le (3+2RD) \cdot \kappa(\Norm x).\]
    Similarly for $\Norm x \leq \Norm y$,
    we have 
    \[(3+2RD)^{-1}\cdot \kappa(\Norm x) \le \kappa(\Norm y)\le  \kappa(\Norm x).\]
    Hence,
    \[(3+2RD)^{-1}\cdot \kappa(\Norm x) \le \kappa(\Norm y)\le (3+2RD) \cdot \kappa(\Norm x).\]
\end{proof}
\begin{definition}[$\kappa$--neighborhood]  \label{Def:Neighborhood} 
For a closed set $Z$ and a constant $\nn$ define the $(\kappa, \nn)$--neighbourhood 
of $Z$ to be 
\[
\calN_\kappa(Z, \nn) = \Big\{ x \in X \ST 
  d_X(x, Z) \leq  \nn \cdot \kappa(x)  \Big\}.
\]

\begin{figure}[h]
\begin{tikzpicture}
 \tikzstyle{vertex} =[circle,draw,fill=black,thick, inner sep=0pt,minimum size=.5 mm] 
[thick, 
    scale=1,
    vertex/.style={circle,draw,fill=black,thick,
                   inner sep=0pt,minimum size= .5 mm},
                  
      trans/.style={thick,->, shorten >=6pt,shorten <=6pt,>=stealth},
   ]

 \node[vertex] (a) at (0,0) {};
 \node at (-0.2,0) {$\go$};
 \node (b) at (10, 0) {};
 \node at (10.6, 0) {$b$};
 \node (c) at (6.7, 2) {};
 \node[vertex] (d) at (6.68,2) {};
 \node at (6.7, 2.4){$x$};
 \node[vertex] (e) at (6.68,0) {};
 \node at (6.7, -0.5){$x_{b}$};
 \draw [-,dashed](d)--(e);
 \draw [-,dashed](a)--(d);
 \draw [decorate,decoration={brace,amplitude=10pt},xshift=0pt,yshift=0pt]
  (6.7,2) -- (6.7,0)  node [black,midway,xshift=0pt,yshift=0pt] {};

 \node at (7.8, 1.2){$\nn \cdot \kappa(x)$};
 \node at (3.6, 0.7){$||x||$};
 \draw [thick, ->](a)--(b);
 \path[thick, densely dotted](0,0.5) edge [bend left=12] (c);
\node at (1.4, 1.9){$(\kappa, \nn)$--neighbourhood of $b$};
\end{tikzpicture}
\caption{A $\kappa$-neighbourhood of a geodesic ray $b$ with multiplicative constant $\nn$.}
\end{figure}
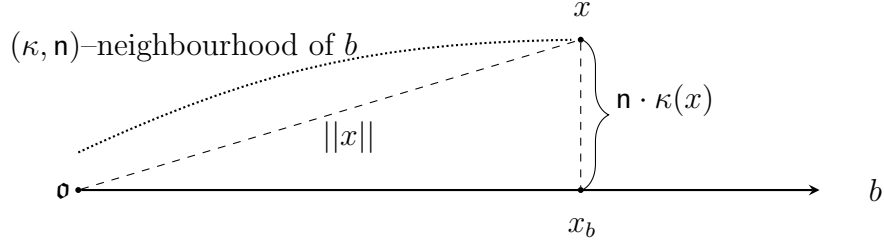
\end{definition}

\begin{definition} [$\kappa$-Morse I, $\kappa$-Morse II] \label{D:k-morse} \label{Def:Morse} 
Let $Z \subseteq X$ be a closed set, and let $\kappa$ be a concave sublinear function. 
We say that $Z$ is \emph{$\kappa$-Morse} if one of the following equivalent (see Proposition 3.10 \cite{QRT2})  condition holds:
\begin{enumerate}

\item[I.] (\emph{$\kappa$-strongly Morse}) There exists a proper function 
$\mm_Z : \mathbb{R}^2 \to \mathbb{R}$ such that for any sublinear function $\kappa'$ 
and for any $r > 0$, there exists $R$ such that for any $(\qq, \sQ)$-quasi-geodesic ray $\beta$
with $\mm_Z(\qq, \sQ)$ small compared to $r$, if 
$$d_X(\beta_R, Z) \leq \kappa'(R)
\qquad\text{then}\qquad
\beta|_r \subset \calN_\kappa \big(Z, \mm_Z(\qq, \sQ)\big)$$

\item[II.](\emph{$\kappa$-weakly Morse})  There is a function $\mm'_Z \from \RR_+^2 \to \RR_+$ so that if $\beta \from [s,t] \to X$ is a $(\qq, \sQ)$--quasi-geodesic with end points 
on $Z$ then
\[
[s,t]_{\beta}  \subset \calN_{\kappa} \big(Z,  \mm'_Z(\qq, \sQ)\big). 
\]
\end{enumerate}
\end{definition} 

\begin{remark}
By taking the maximum function of $\mm_Z, \mm'_Z,$ we may and will always assume that both conditions hold for the same $\mm_Z$, which we refer to as the $\kappa$-Morse gauge. Further,
\begin{equation} 
\mm_Z(\qq, \sQ) \geq \max(\qq, \sQ). 
\end{equation} 
\end{remark}

\begin{definition} \label{weakprojection}
Let $(X, d_X)$ be a proper geodesic metric space and $Z \subseteq X$ a closed subset, and let 
$\kappa$ be a concave sublinear function. A map $\pi_{Z} \from X \to \mathcal{P}(Z)$ is a $\kappa$-\emph{projection}  if there exist constants $D_{1}, D_{2}$, depending only on $Z$ and $\kappa$, such that for any points $x \in X$ and $z \in Z$, 
\[
\diam_X(\{z \} \cup \pi_{Z}(x)) \leq D_{1} \cdot d_X(x, z) + D_{2} \cdot \kappa(x).
\]
\end{definition}

A $\kappa$-projection differs from a nearest point projection by a uniform multiplicative error and a sublinear additive error. In particular, the nearest point projection is a $\kappa$-projection. Indeed, for $z \in Z$ and $w \in \pi_Z(x)$, we have
\[
d(z, w) \leq d(z, x) + d(x, w) \leq 2d(z, x).
\]
\begin{theorem}\label{A-sublinearlyequivalence}
Let $(X, \go)$ be a proper geodesic metric space with a fixed basepoint and let $\alpha$ be a quasi-geodesic
ray in $X$. Let $\pi$ be any $\kappa$-projection from $X$ to $\alpha$ in the sense of Definition~\ref{weakprojection}.
Then if $\alpha$ is $\kappa$-weakly Morse, then it is $\kappa'$-weakly contracting with respect to $\pi$ for some sublinear function $\kappa'$.
\end{theorem}

\begin{definition}[sublinear equivalence classes in $\pka X$] \label{Def:Fellow-Travel}
Let $\beta$ and $\gamma$ be two quasi-geodesic rays in $X$. If $\beta$ is in some 
$\kappa$--neighbourhood of $\gamma$ and $\gamma$ is in some 
$\kappa$--neighbourhood of $\beta$, we say that $\beta$ and $\gamma$ 
\emph{$\kappa$--fellow travel} each other. This defines an equivalence
relation on the set of quasi-geodesic rays in $X$ (to obtain transitivity, one needs to change $\nn$ of the associated $(\kappa, \nn)$--neighbourhood). 
\end{definition}

We denote the equivalence class 
that contains $\beta$ by $[\beta]$. We say two quasi-geodesic rays $\alpha, \beta$ \emph{sublinearly fellow-travel} if as $r \to \infty$, 
\[\lim_{r \to \infty} \frac{d(\alpha_r, \beta_r)}{r} = 0.\]

It is established in \cite{QRT2} that suppose $\alpha$ is a $\kappa$-Morse quasi-geodesic ray and suppose $\beta$ is a sublinearly Morse quasi-geodesic ray that sublinearly fellow travels $\alpha$ then $\beta$ $\kappa$-fellow travels $\alpha$. Thus it is a generalization of Gromov hyperbolicity in the sense that in these set of directions sublinear fellow traveling implies uniform sublinearly fellow travel. 

\begin{definition}[Sublinearly Morse boundary]
Let $\kappa$ be a sublinear function as specified in Section~\ref{subfunction} and let $X$ be a proper geodesic space.
\[\pka X : = \{ \text{ all } \kappa\text{-Morse quasi-geodesics } \} / \text{sublinear-fellow traveling}\]
\end{definition}

\subsubsection{A coarse cone topology on $\pka X$}\label{coarsetop}

 We equip $X$ with a topology which is 
a coarse version of the visual topology. In visual topology, if two geodesic rays fellow travel for a long time, then they are ``close''. In this coarse version, if two geodesic rays and all the quasi-geodesic rays in their respect equivalence classes remain close for a long time, then they are close. Now we define it formally. First, we say a quantity $\sD$ \emph{is small compared to a radius $\rr>0$} if 
\begin{equation} \label{Eq:Small} 
\sD \leq \frac{\rr}{2\kappa(\rr)}. 
\end{equation} 

\begin{definition}[Topology on $\pka X$]\label{openset}
Let $\bfa \in X$ and $\alpha_0 \in \bfa$ be 
the unique geodesic in the class $\bfa$. Define $\calU_{\kappa}(\bfa, \rr)$ to be the set of points $\bfb$ such that for any $(\qq, \sQ)$-quasi-geodesic of $\bfb$, denoted  $\beta$, 
such that $\mm_{\beta}(\qq, \sQ)$ is small compared to $\rr$, satisfies 
\[
\beta|_{\rr} \subset \calN_{\kappa}\big(\alpha_{0}, \mm_{\alpha_{0}}(\qq, \sQ)\big).\]

Let the topology of $\pka X$ be the topology induced by this neighbourhood system. The following fact shows that a $\kappa$-boundary is well defined with respect the associated group.
\end{definition}

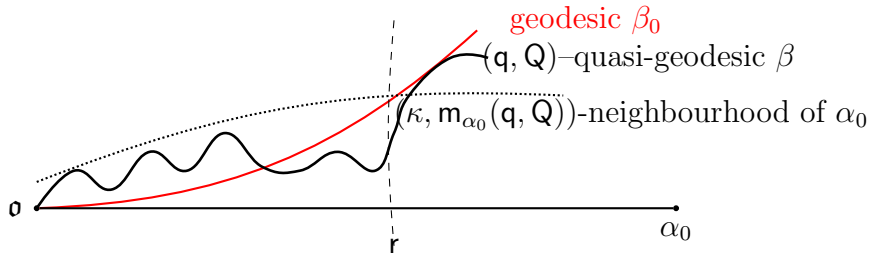
\begin{figure}[H]
\begin{tikzpicture}[scale=0.5]
 \tikzstyle{vertex} =[circle,draw,fill=black,thick, inner sep=0pt,minimum size=.5 mm]
 
[thick, 
    scale=1,
    vertex/.style={circle,draw,fill=black,thick,
                   inner sep=0pt,minimum size= .5 mm},
                  
      trans/.style={thick,->, shorten >=6pt,shorten <=6pt,>=stealth},
   ]

  \node[vertex] (o) at (-10,0)  [label=left:$\go$] {}; 
  \node (o1) at (2, 5)[label=right:\color{red} geodesic $\beta_{0}$] {}; 
  \draw[thick, red]  (o) to [ bend right=20] (o1){};
  \node[vertex] (a) at (7,0)  [label=below:$\alpha_{0}$] {};
 \path[thick, densely dotted](-10,0.7) edge [bend left=12] (4, 3){};
 \node at (5.8, 2.5){$(\kappa, \mm_{\alpha_{0}}(\qq, \sQ))$-neighbourhood of $\alpha_{0}$};

       \node (a1) at (-1, 5) {}; 
   \node at (6,4) { $(\qq, \sQ)$--quasi-geodesic $\beta$};  
  \draw[dashed]  (-0.5, 5) to  [bend right=5] (-0.5, -1){};
  \node at (-0.5, -1){$\rr$};

   \draw[thick]  (o)--(a){};

  \pgfsetlinewidth{1pt}
  \pgfsetplottension{.75}
  \pgfplothandlercurveto
  \pgfplotstreamstart
  \pgfplotstreampoint{\pgfpoint{-10cm}{0cm}}  
  \pgfplotstreampoint{\pgfpoint{-9cm}{1cm}}   
  \pgfplotstreampoint{\pgfpoint{-8cm}{0.5cm}}
  \pgfplotstreampoint{\pgfpoint{-7cm}{1.5cm}}
  \pgfplotstreampoint{\pgfpoint{-6cm}{1cm}}
  \pgfplotstreampoint{\pgfpoint{-5cm}{2cm}}
  \pgfplotstreampoint{\pgfpoint{-4cm}{1.1cm}}
  \pgfplotstreampoint{\pgfpoint{-3cm}{1cm}}
  \pgfplotstreampoint{\pgfpoint{-2cm}{1.5cm}}
  \pgfplotstreampoint{\pgfpoint{-1cm}{1cm}}
    \pgfplotstreampoint{\pgfpoint{-0.5cm}{2cm}}
      \pgfplotstreampoint{\pgfpoint{-0.1cm}{3cm}}
        \pgfplotstreampoint{\pgfpoint{1cm}{4cm}}
          \pgfplotstreampoint{\pgfpoint{2cm}{4cm}}
  \pgfplotstreamend 
  \pgfusepath{stroke}
       
  \end{tikzpicture}
 
  \caption{ $\bfb \in \calU_{\kappa}(\bfa, \rr)$ because the quasi-geodesics of $\bfb$ such as $\beta, \beta_{0}$ stay inside the associated $(\kappa, \mm_{\alpha_{0}}(\qq, \sQ))$-neighborhood of $\alpha_{0}$ (as in Definition~\ref{Def:Neighborhood}), up to distance $\rr$. }
 \end{figure}

\subsection{First passage percolation} \label{fpp}
 We consider a connected non-oriented graph $X$, whose set of vertices (resp. edges) is denoted by $V$ (resp. E). For every function $\omega: E \rightarrow(0, \infty)$, we equip $V$ with the weighted graph metric $d_{\omega}$, where each edge $e$ has weight $\omega(e)$. In other words, for every $v_{1}, v_{2} \in V, d_{\omega}\left(v_{1}, v_{2}\right)$ is defined as the infimum over all path $\gamma=\left(e_{1}, \ldots, e_{m}\right)$ joining $v_{1}$ to $v_{2}$ of $|\gamma|_{\omega}:=\sum_{i=1}^{m} \omega\left(e_{i}\right)$. Observe that the graph metric on $V$ corresponds to the case where $\omega$ is constant equal to 1 , we shall simply denote it by $d$. We consider a probability measure on the set of all weight functions $\omega$. We let $\nu$ be a probability measure supported on $[0, \infty)$. Our model consists in choosing independently at random the weights $\omega(e)$ according to $\nu$. More formally, we equip the space $\Omega=[0, \infty)^{E}$ with the product probability that we denote by $P$. In this project we always assume that $\nu({0})=0$, but generally we do not assume that the support of $\nu$ is bounded away from zero.

 Let $d(\gamma, \go)$ denote the distance between $o$ and the set of vertices $\{\gamma(0), \gamma(1), \ldots\}$. That is 
 \[d(\gamma, \go) : = \inf\bigg\{d(v, \go) : v \in \{\gamma(0), \gamma(1), \ldots\}  \bigg\} \]
The following two results (Proposition~\ref{BT1} and ~\ref{extension}) along with the lemma ~\ref{lemma2.13} established upper and lower bounds for how distance can be perturbed under first passage percolation. Proposition ~\ref{BT1} is established in \cite{BT17}. Proposition \ref{extension} is a generalization of  Lemma 2.5 in \cite{BT17}. 

\begin{proposition}\cite{BT17}\label{BT1}
Let $X$ be a connected graph, and let $\gamma$ be a self-avoiding path. Assume $0<b=\mathbb{E} \omega_{e}<\infty$. Then for a.e. $\omega$, there exists $r_{0}=r_{0}(\omega)$ such that for all $i \leq 0 \leq j$,

$$
\ell_{\omega}(\gamma([i, j]) \leq 2 b(j-i)+r_{0}
$$
\end{proposition} 
\begin{lemma}[Lemma 2.4 \cite{BT17}] \label{lemma2.13} Let $X$ be an infinite connected graph with bounded degree and assume that $\nu({0})=0.$ There exists an increasing function $\alpha:(0,\infty)\to (0,1]$ such that $\lim_{t\to0}\alpha(t) = 0,$ and such that for all finite path $\gamma$ and all $\epsilon>0,$
$$\mathbb{P}\big(\ell_\omega(\gamma)\leq \epsilon\cdot \ell(\gamma)\big) \leq \alpha(\epsilon)^{\ell(\gamma)}.$$
\end{lemma}

\begin{proposition}\label{extension}\label{BT2}
 Let $X$ be an infinite connected graph with bounded degree, and let $\go$ be some vertex of $X$. Assume $\nu(\{0\})=0$. Then for every constant $D$, there exists $c>0$ such that for a.e. $\omega$, there exists $r_{1}=r_{1}(\omega)$ such that for all finite path $\gamma$ such that $d(\gamma, o) \leq D \cdot \ell(\gamma)$, one has

\[ \ell_\omega(\gamma)\geq c(D) \cdot \ell(\gamma)-r_{1}\]
Furthermore, $c(\param)$ is an decreasing function.
\end{proposition}
\begin{proof}
Let $q$ be an upper bound on the degree of $X$, for simplicity, suppose $q\geq 2$ and let $n\geq 1$ be some integer. Suppose $\gamma$ is a path satisfying $d(o,\gamma) \leq D \cdot \ell(\gamma) = Dn$. Then the number of vertices in the ball $B(\go, Dn)$ is, 
\[|B(\go, Dn)| \leq 1+q \cdot \displaystyle\sum_{i=1}^{Dn-1}(q-1)^{i}\leq \displaystyle\sum_{i=0}^{Dn-1}(q+1)^i \leq (q+1)^{Dn+1}.\]Now possible number of paths of length $n$, starting in the ball $B(\go,Dn)$ is \[P_{n,D} \leq q\cdot (q-1)^{n-1}\cdot (q+1)^{Dn+1} \leq (q+1)^{(D+1)n+1} \]
Suppose $\alpha(t)$ be the increasing function with $\lim_{t\to 0} \alpha(t)=0$ in the Lemma ~\ref{lemma2.13}. Now define a function $\beta:(0,1] \to (0,\infty)$ such that,
\[\beta(t):= \sup \{x \in(0,\infty) : \alpha(x) \leq t \}.\] 
The function  is well defined since for every $t$ there exists $\delta_t>0$ such that $\alpha(x) \leq t$ wherever $x \in [0,\delta_t]$. Therefore, $[0,\delta_t] \subset \{x \in(0,\infty) : \alpha(x) \leq t \}$. Moreover, $\beta$ is non-decreasing and right continuous. Let, $C(D)$ be a function defined as $$C(D)=\beta(\frac{1}{q^{D+1}})$$
As a function of $D$, $C(D)$ is decreasing. Moreover by definition, $\alpha(C(D)\leq \frac{1}{q^{D+2}}$. Let us consider the following event, 
\[
A_{i,n,D}: = \{ \text{there exists some path $\gamma_i$ of length $n$, with $d(o,\gamma_i) \leq Dn$, and $\ell_\omega(\gamma_i)\leq c(D)\cdot \ell(\gamma_i)$\}}.\] 
By Lemma ~\ref{lemma2.13} $\mathbb{P}(A_{i,n,d}) \leq \alpha(C(D))^n\leq \frac{1}{q^{Dn+n}}$. There are up to $(q+1)^{(D+2)n}$ possible number of paths. Hence, the probability of at least one of these paths satisfies this property is, 
\begin{align*}
\mathbb{P}\left(\bigcup_{i=1}^{(q+1)^{(D+1)n+1}} A_{i,n, D}\right) &\leq \sum_{i=1}^{(q+1)^{(D+1)n+1}} \mathbb{P}m(A_{i,n, D})\\
&\leq (q+1)^{(D+1)n+1} \cdot \frac{1}{(q+1)^{Dn+2n}} \\
&= \frac{1}{(q+1)^{n-1}}.
\end{align*}
Thus $ \sum_{n=1}^\infty \frac{1}{(q+1)^{n-1}} < \infty$. Therefore, by the Borel-Cantelli lemma, the probability that this event happens infinitely often is 0.
\end{proof}
\begin{corollary}
Let $X$ be an infinite connected graph with bounded degree, and let $o$ be some vertex of $X$. Assume $\nu(\{0\})=0$. Let $\gamma:[0,n]\rightarrow X$ be any finite self avoiding path of length $n$, such that $ d(\gamma, o) \leq D \cdot n$. Then the image of $\gamma$ in $X_\omega$, is a quasi-isometric image of $\gamma$ whose constants depends only on $\omega$, i.e. there exists constants $q(\omega),Q(\omega)$ such that,
\[
 \frac{1}{q} \cdot \ell(\gamma)-Q(\omega)\leq \ell_{\omega}(\gamma) \leq q \cdot \ell(\gamma)+Q(\omega)
\]

\end{corollary}
\begin{proof}
Let $Q(\omega):= \max \{ r_{1}(\omega), r_{0}(\omega) \}$ and let $q: = \max \{2b,\frac{1}{c(D)}\}$ then by Propositions ~\ref{BT1}, ~\ref{BT2} we have the desired inequalities both sides.
\end{proof}

Now take the union of all $\kappa$-Morse boundaries and it is established in \cite{QRT2} that the topology on them are consistent with each other and in fact there exists a natural topology on the set of all sublinearly Morse directions, of which all $\kappa$-boundaries are subspaces with subspace topology.
\begin{definition}
Let $\partial_s X$ denote the set of all sublinearly Morse quasi-geodesic rays, up to sublinear fellow travelling. We refer to this as the \emph{Sublinearly Morse boundary}.
\end{definition}
\begin{definition}
The topology on $\partial_s X$ is defined as follows: let $\beta$ be a $\kappa$ Morse sublinearly quasi-geodesic ray for some $\kappa$. An equivalence class $\bfa \in 
\ps X$ belongs to $U(\beta,r)$ if, for any $(q,Q)$-quasi-geodesic ray $\alpha \in \bfa$,
where $m_\beta(q,Q)$ is small compared to $r$ with respect to $\kappa$, we have the inclusion
\[ \alpha|r \ \subset N_\kappa( \beta,m_\beta(q,Q)).\]
\end{definition}
An open neighborhood $\calV$ of a point $\bfb \in \ps X$ contains all $\kappa'$-Morse quasi-geodesic rays that in the ball of radius $r$ around quasi-geodesic rays of $\bfb$, behaves as if they are a member of $\bfb$. It is established in \cite{QRT2} that with this topology the space $\ps X$ is metrizable and contains each $\pka X$ as a topological subspace.


\section{middle-recurrence of sublinear lines}
In this section we study the middle-recurrence property of sublinearly Morse bi-infinite lines and use it to characterize sublinearly Morse directions and bi-infinite lines. Let $\gamma$ be a sublinearly Morse set. The middle recurrent property concerns the location of paths with endpoints on $\gamma$ as opposed to the location of quasi-geodesic segments with endpoints on $\gamma$. 

By a \textit{path} we mean the continuous image of an interval $[a, b]$ in $X$. In this paper, we assume all paths to be an embedding of an interval and thus are  self-avoiding. This assumption makes the argument shorter.
Given a path $\alpha$ with endpoints $x, y$, suppose $\ell(\alpha)$ be the arc-length of $\gamma$ then  the \textit{slope} of $\alpha$ is defined to be 
\[ sl(\alpha):= \frac{\ell(\alpha)}{d(x,y)}  \]
where the arc-length of a continuous curve $\alpha: I \rightarrow X$ is defined as the infimum of the sum of the distances $d(\alpha(t_{i-1}), \alpha(t_i))$ over any partition $\{t_i\}_{i\geq0}$ of $I$.
\begin{definition}
Let $\gamma$ be set. We say that a $\gamma$ has the \textit{middle-recurrence property}  if for every $C\geq1$ there is a constant $c$ (depends on $C$ and $\beta$) such that the following holds: 
    for any path $P$ with endpoints $a,b \in \gamma$ satisfying $\ell(P)\leq Cd(a,b)$, intersects the $c\cdot \kappa$-neighborhood of the middle third of $\gamma[a,b]$ for some sublinear function $\kappa$. That is,
\[ P \cap \calN_{\kappa} (\gamma_{\frac 13[a,b]}, c )\neq \emptyset.\]

\end{definition}
Middle recurrence is first studied by \cite{DMS} and then corrected and further expanded by \cite{ADT}. Here we use an distance estimation lemma \cite[Lemma 3.3]{ADT} that gives an upper bound  of the distance between the end points of a path with fixed slope and other properties. Let $p$ be a path with end points $s$ and $e$. 

\begin{lemma}
Let $\beta$ be a $\kappa'$–radius-contracting quasigeodesic and suppose that $p$ is a path with the following properties:
\begin{itemize}
\item slope of $p$ is $D$
\item $p$ is at distance at least $K$  from $\beta$
\item The end points of $p$ are at distance exactly $K$ from $\beta$. 
\end{itemize}
Let $|p|: d= d(s, e)$ denote the distance between the endpoints of $p$. Then
\[ \frac{\kappa(K)}{2K} \geq \frac{1-(2K+\kappa'(K))/|p|}{D}\]
\end{lemma}

\subsection*{Notation}In the rest of this paper, if $\gamma$ is a geodesic line, then the restriction of $\gamma$ between two points $a, b\in \gamma$ is notated as $\gamma{[a, b]}$. Furthermore, if we are referring to a segment of this geodesic segment that starts at one third from the endpoint $a$ and ends at $\frac 23$ from the endpoint $a$ then we notate it as $\gamma_{\frac13 [a,b]}$. In general, if a geodesic segment $\alpha$ lies on a geodesic ray that emanates from the origin, then it is assumed that $\alpha_{[p, q]}$, where $0\leq p, q \leq 1$, lies between the fractions $p$ and $q$ starting from the point closer to the origin.

In \cite{JanaQing1}(Theorem A]) we have proved the following one sided Proposition:
\begin{proposition}[$\kappa$-Morse implies middle recurrence] \label{prop3.3}
Let $\gamma$ be a bi-infinite geodesic line. Assume that  $\gamma$ is $\kappa$-Morse with Morse gauge $M$. Let $P$ be any simple path with endpoints $a, b \in \gamma$ and such that $\ell(P) < C d(a, b)$. Then there exists a sublinear neighborhood of $\gamma$ depending only on $\gamma$ and $C$ such that 
\[ P \cap \calN_{\kappa'} (\gamma_{\frac13 [a,b]}, c(C, \beta)) \neq \emptyset.\]
Here $\kappa'$ is the contracting sublinear function of $\beta$.
\end{proposition}
Now in the following proposition we will establish the other direction:
\begin{proposition}[Middle recurrence implies sublinearly Morse] \label{prop3.4}
 Let $\gamma$ be a geodesic ray. Suppose  there exist function $M \from \mathbb{R_{+}}\to\mathbb{R_{+}}$ and a sublinear function $\kappa$ such that the following holds:  
 \begin{itemize}
     \item for every self-avoiding path $P$ with slope $sl(P)<C$ and endpoints $a,b\in\gamma$,
\[
P \cap \mathcal{N}_\kappa \left(\gamma_{\frac13 [a,b]},\, M(C)\right)\neq \emptyset.
\]
 \end{itemize}
 Then $\gamma$ is sublinearly Morse. 


\end{proposition}
Before giving the main proof of the proposition, we establish the following lemma. 
\begin{lemma}\label{middle rec quasi geodesic} (See figure ~\ref{middle rec lemma})
Let $\gamma$ be a geodesic and assume the hypotheses of Proposition \ref{prop3.4}. Let $p:I\to X$ be a continuous $(q,q)$--quasi-geodesic with endpoints $a,b\in\gamma$.
Fix
\[
d:=\max\left\{M\left(\frac{13q+2}{10},0\right),\,1\right\}.
\]
If $\pp'\subset P$ is a sub path that remains at most distance $d\cdot\kappa(t)$ from $\gamma$ and its endpoints
$a'$ and $b'$ satisfy,
$
d(a',b') \geq 12d \cdot \max \{ \kappa(\Norm{a'},  \kappa(\Norm{b'} \}
$
then
\[
\pp'\cap \mathcal N_{\kappa}(\gamma,d)\neq\emptyset.
\]
\end{lemma} 
\begin{remark}
  Although we stated the lemma for a $\pp$ being a $(q,q)$-quasi-geodesic, the conclusion holds for an arbitrary quasi-geodesic: by the taming lemma (\cite{BH1}, Lemma III.H.1.11]), any quasi-geodesic can be replaced by a tame one (with controlled constants).  
\end{remark}
\begin{figure}[H] 

\includegraphics[height=5cm,width=13cm]{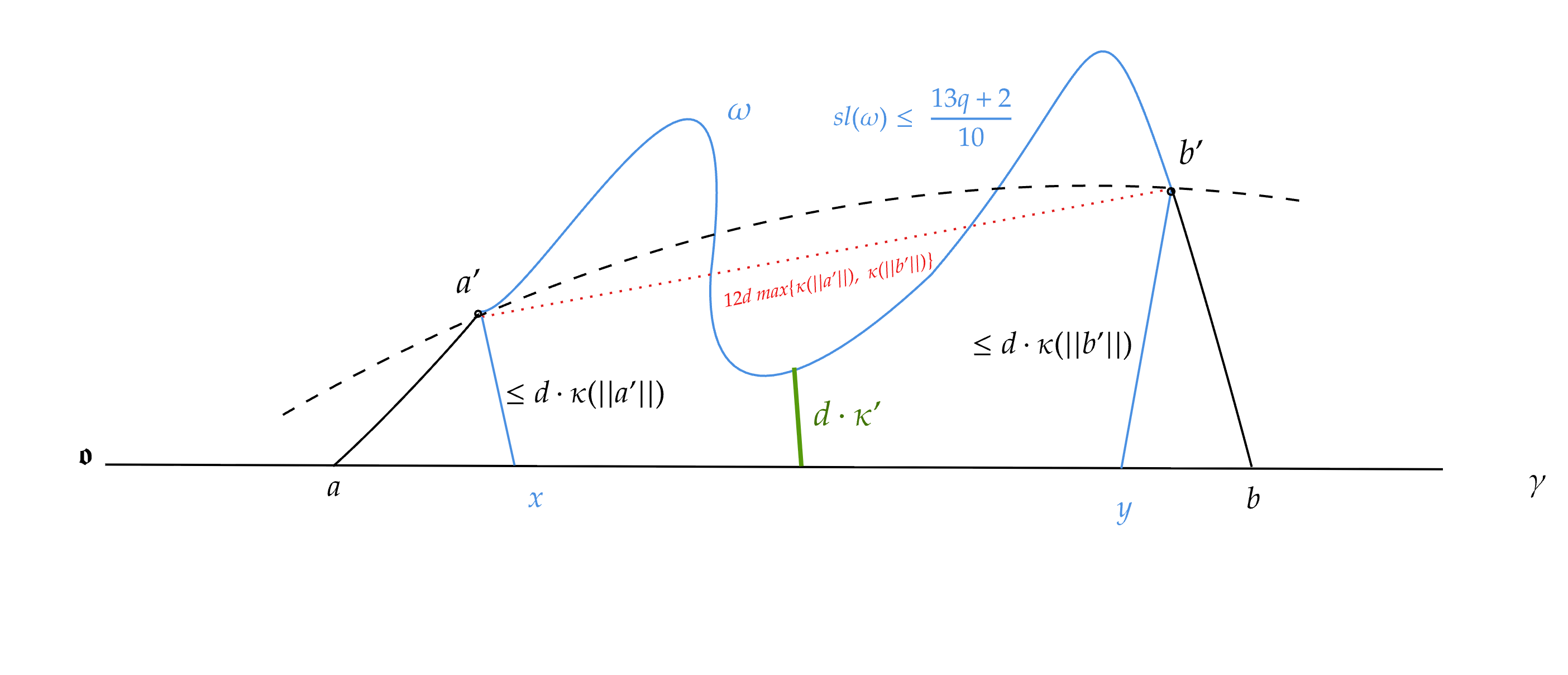} 
\caption{ if $a',b'$ are $\kappa-$separated then a sub path $\pp'$ intersects a $\kappa'$ neighborhood of $\gamma.$ \label{middle rec lemma}} 
\end{figure} \begin{proof}
Suppose, $A=A(a',b')=  \max \{\kappa(||a'||), \kappa (||b'||)\}$. Let \( x, y \in \gamma \) such that \( d(x, a') \leq d \cdot \kappa(||a'||) \), \( d(y, b') \leq d \cdot \kappa(||b'||)\). Let \( \omega \) be the concatenation \( [x, a'] \cdot \pp' \cdot [b', y] \). From the given hypothesis, we have the following estimates:
\begin{equation}
   d(x, y)\geq d(a',b') - d(a',x) -d(b'y) \geq 12d \cdot A(a',b')-d \kappa(||a'||)-d \kappa(||b'||)  \label{eq2}
\end{equation}
Therefore we have,
\begin{equation}
 d(x, y)\geq 10 d \cdot A(a',b') \label{eq3}
\end{equation} On other hand, since $d(a',b)\geq 12d \cdot A(a',b')$ we get,
\begin{equation}
d(a',b') - 2d \cdot A(a',b') \ge \frac{5}{6}\,d(a',b') .
\label{eq4}
\end{equation}
 Since $\pp$ is $(q,q)$ quasi-geodesic we can assume that $\Norm{\pp'} \leq q \cdot d(a',b') +q$.
Now we compute the slope of $\omega$:
\begin{align*}
\frac{\ell(\omega)}{d(x, y)} &\leq \frac{\Norm{\pp'}+ d(a',x)+d(b',y)}{d(x,y)}\\
&\leq\frac{\Norm{\pp'}+ d\cdot \kappa(\Norm{a'})+ d \cdot \kappa(\Norm{b'})}{d(x,y)}\\
&\leq \frac{\Norm{\pp'}}{d(a',b')-d \cdot \kappa(\Norm{a'})-d \cdot \kappa(\Norm{b'})}+ \frac{d\cdot \kappa(\Norm{a'})+ d \cdot \kappa(\Norm{b'})}{d(x,y)} \\
&\leq \frac{q \cdot d(a',b')+q}{d(a',b')-2d \cdot A(a',b')}+ \frac{2d \cdot A(a',b')}{10d \cdot A(a',b')} \qquad\text{(by equation~\ref{eq3})}\\ 
&\leq \frac{q \cdot d(a',b')+q}{\frac{5}{6}d(a',b')}+ \frac{1}{5}  \qquad\text{(by equation~\ref{eq4})}\\
&\leq \frac{13q+2}{10}
\end{align*}
Thus, $\omega$ has bounded slope. Hence, by middle recurrence, there is a point $c \in \pp'$ that intersects the $\kappa$-neighborhood of the middle third of $\gamma_{[x, y]}$ i.e. there exist $z \in \gamma_{\frac{1}{3}[x, y]}$ such that $d(c,z) \leq d \cdot \kappa(\Norm z).$
\newline \underline{\emph{Claim:} } $c$ belongs to $\pp'.$
\newline \underline{\emph{proof of the claim:}} Suppose without loss of generality $c$ is on the geodesic joining $a'$ and $x$. By definition ~\ref{def:middle recurrent} and equation ~\ref{eq3} we have, 
\begin{equation}
    3d(z,x) \geq  d(x,y) \geq 10d \cdot A(a',b'). \label{eq5}
\end{equation}
Note that, $\Norm c \leq \Norm {a'}+ d(a',x) \leq 2 \Norm {a'}.$ Hence, by triangle inequality we have, 
\begin{align*}
  3d(z,x) &\leq d(z,c)+d(c,x)\\
  & \leq d \cdot \kappa(\Norm c)+d \cdot \kappa(\Norm a')\\
  & \leq 9d\cdot A(a',b')
\end{align*}
which is a contradiction to equation ~\ref{eq5}. Hence, the lemma follows.

\end{proof}
\subsection*{Proof of Proposition ~\ref{prop3.4}}

Let, $Q =\{t\in I:d(\pp(t), \gamma) \leq d \cdot \kappa(\Norm{\pp(t)})\}$. $Q^c$ is a collection of disjoint unions of open intervals in $\mathbb{R}$. If $(t,t') \in Q^c$ then \[d (\pp(t),\pp(t')) < 12d \cdot A(\pp(t),\pp(t')).\] If $\eta \in Q^c$ then $\eta \in (t,t')\subset Q^c$ for some $t,t' \in \mathbb{R}$. We have then,
 \[d(\pp(\eta), \pp(t))< d \cdot  \max \{ \kappa(\Norm {\pp(\eta)}, \kappa(\Norm{\pp(t))}.\} \]
 Otherwise it contradicts Lemma ~\ref{middle rec quasi geodesic}.
 Therefore, \[d(\pp(\eta),\gamma) < d(\pp(\eta), \pp(t))+ d \cdot \kappa(\Norm{\pp(t))} < 12d \cdot  A(\pp(\eta), \pp(t))+d \cdot \kappa(\Norm{\pp(t)}).\]
 Since $ d(\pp(\eta), \pp(t)) \leq 12d \cdot \max \{ \kappa(\Norm {\pp(\eta)}, \kappa(\Norm{\pp(t))}\}$ then by Lemma ~\ref{lemma2.3} 
 there exist $D >0$ such that, 
 $$\kappa(\Norm {\pp(t))} \le D \cdot \kappa(\Norm{\pp(\eta))}$$ 
Therefore, \[d(\pp(\eta),\gamma) \leq 12d\cdot (\kappa(\Norm {\pp(\eta)})+D\cdot \kappa(\Norm \pp(\eta))+ d \cdot D\cdot \kappa(\Norm {\pp(\eta)}).\]
This implies, \[d(\pp(\eta),\gamma) \leq d \cdot (13 D+12) \cdot  \kappa(\Norm {\pp(\eta)}).\]
Thus, there exists a sublinear function $\kappa'$ depends on $\kappa$ such that $\pp \subset \calN_{\kappa'}(d, \gamma)$. Hence, Proposition ~\ref{prop3.4} follows.
\hfill $\blacksquare$\\
Now we can conclude the following theorem, as Proposition ~\ref{prop3.3}, Proposition ~\ref{prop3.4} establish the statement in each direction.
\begin{theorem}
Let $X$ be a proper geodesic metric space. Let $\beta$ be a $\qq$-quasi-geodesic ray. Then $\beta$ is sublinearly Morse if and only if $\beta$ is $\kappa$-middle recurrent for some $\kappa$.
\end{theorem}

\section{Sublinearly Morse directions under first passage percolations}
Now we are equipped with the path charaterization of sublinearly Morse directions, we are ready to establish Theorem~\ref{introthm1} in this final section. Recall from Section~\ref{fpp} all relevant background of first passage percolation. In particular recall that we always assume $\nu({0})=0$. The first lemma shows that sublinear neighborhoods are preserved.
\begin{lemma}\label{neighborhood}
Let $X$ be a graph of bounded degree with a fixed base point. Let $\gamma$ be any geodesic ray in $X$ emanating from the base point. Then the set of $\omega$ that maps any sublinear neighborhood of any $\gamma$ in $X$ to a linear neighborhood in $X_\omega$ of $\gamma_\omega$ is a set of measure zero.
\end{lemma}
\begin{proof}
Suppose $\calN_\kappa(\gamma,n)$ be a sublinear neighborhood of $\gamma$ in $X$ for some sublinear function $\kappa$ and gauge $n$. Let $x \in \calN_\kappa(\gamma,n)$. Then by Proposition ~\ref{BT1}, \[d_\omega(x, \gamma) \leq b \cdot n \cdot \kappa(||x||) + r_o(\omega).\]
Consider $[\go,x]_\omega$ be the $\omega$-geodesic in $X_\omega$. Then the preimage of  $[\go,x]_\omega$ in $X$ satisfies the assumption of Proposition ~\ref{BT2}. Therefore, we get \[||x||_\omega=\ell_\omega([\go,x]_\omega) \geq c \cdot ||x||-r_1(\omega).\]
Hence for a.e. $\omega$ we have, 
\begin{align*}
d_\omega(x,\gamma) \leq  & \frac{b\cdot n}{c} \cdot \kappa(||x||_\omega+r_1(\omega))+r_0(\omega)\\
& \leq \bigg(\frac{b\cdot n}{c}+r_0(\omega)\bigg) \cdot \kappa'(||x||), \text{ where, $\kappa'(t)= \kappa(t+r_1(\omega))$.}
\end{align*}
Therefore, almost surely, FPP sends a sublinear neighborhood of a geodesic ray in $X$ to a sublinear neighborhood of the image ray in $ X_\omega$.


\end{proof}
\begin{theorem} \label{37}
The image of a $\kappa$-Morse geodesic is almost surely sublinearly Morse. Specifically, let $\gamma$ be a $\kappa$-Morse geodesic ray emanating from a fixed base point $\go$ in $X$. Then for a.e. $\omega$ there exists $\kappa'$, depending on $\kappa$ and $\omega$ such that $\gamma_\omega$ is $\kappa'$-Morse.
\end{theorem}
\begin{proof}
Assume $\gamma^{\omega}\subset (X,d_{\omega})$ is not sublinearly Morse. Then there exist $(q,Q)$, a constant $c_{1}>0$, and a family of $(q,Q)$--quasi--geodesic segments 
$\{\alpha_i\}_{i\in\mathbb N}$ with endpoints $(s_i,e_i)$ on $\gamma^{\omega}$ such that for each $i$ there exists a point 
$x_i\in \alpha_i$ satisfying
\[
d_{\omega}(o,x_i)\ge i
\qquad\text{and}\qquad
d_{\omega}(x_i,\gamma^{\omega})\ge c_{1}\, d_{\omega}(o,x_i).
\]
In particular, $\alpha_i$ is not contained in the $c_{1}$--linear neighborhood of $\gamma^{\omega}$.


Now consider the pre images of $\alpha_i$ (say, $\alpha_i^X)$ in $X$. For all $i$, the endpoints of $\alpha_i^X$ on $\gamma$ are $s_i,e_i$ respectively. We make the following claim:
\begin{claim}
Up to an infinite subsequence $\{i_k\}$, the base of $\alpha^X_{i_k}$ are bounded below by $c_3 \cdot d(\go, s_{i_k})$ for some $c_3$, i.e. there exists $c_3$ such that $d(s_{i_k},e_{i_k}) \geq c_3 \cdot d(s_{i_k}, \go). $ 
\end{claim}
\begin{proof} From the assumption we have $x_i \in  \alpha_i$ such that, \[d_\omega(x_i,\gamma_\omega) \ge c_1 \cdot d_\omega(\go,x_i).\]
On the other hand,  $d_\omega(\go,s_i) \le d_\omega(\go,x_i) + d_\omega(x_i,s_i) \le \big(\frac{1+c_1}{c_1}\big)\ d_\omega(x_i,s_i).$
 \text{Therefore,}
\begin{equation} \label{eq:upperbound}
   \ell_\omega(\alpha_i) \ge \frac{c_1}{1+c_1} \cdot d_\omega(\go,s_i).
\end{equation}

By Proposition ~\ref{BT2}, for a.e. $\omega$ we get,
\begin{equation} \label{eq:lowerbound}
    d_\omega(\go,s_i)\ge c \cdot d(\go,s_i)-r_1(\omega).
\end{equation}
    
Since, $\alpha_i$ is a $(q,Q)$ quasi-geodesic, we have,
\begin{equation} \label{eq:lengthbound}
 q \cdot d_\omega(s_i,e_i)+Q \ge \ell_\omega(\alpha_i).   
\end{equation}
Now, by combing Equations ~\ref{eq:upperbound}, ~\ref{eq:lowerbound}, ~\ref{eq:lengthbound} and using Proposition ~\ref{BT1}, there exists two constants $L_1,L_2$ (which depends on $\omega$) such that, \[d(s_i,e_i) \ge L_1 \cdot d(\go,s_i)-L_2.\]
By assumption, $d(s_i,\go) \rightarrow \infty$ as $i \rightarrow\infty.$ Therefore, for large $i$, we have \[d(s_i,e_i) \ge c_3\cdot d(\go,s_i).\]

\end{proof}

\begin{claim} \label{39}
The paths $\alpha^X_i$ have bounded slope.
\end{claim}
\begin{proof}
Suppose otherwise, that as $i \to \infty$, the slope of $\alpha^X_i$ tends to infinity i.e. $\frac{\ell(\alpha_i^X)}{d(s_i,e_i)} \rightarrow \infty$. Therefore, there exists $R> 0$ such that $\ell(\alpha_i^X) \ge \frac{1}{c_3}\cdot d(s_i,e_i)$   Since the base grows linearly with respect to $d(\go, s^X_i)$, the lengths of the curves grows superlinearly with respect to the base and distance to the origin. Thus all but finite of $\{ \alpha^X_i \} $ satisfies the condition:
\[ d_X(\go, s^X_i) \leq \ell(\alpha^X_i).\]
By Proposition ~\ref{BT2}, for a.e. $\omega$ we have,
\begin{equation} \label{eq6}
 \ell_\omega(\alpha_i) \ge c \cdot \ell(\alpha_i^X) -r_1
\end{equation}
Hence by Equation ~\ref{eq6} and Proposition ~\ref{BT1}, \[sl(\alpha_i) = \frac{\ell_\omega(\alpha_i)}{d_\omega(s_i,e_i)} \geq \frac{ c \cdot \ell(\alpha_i^X) -r_1}{b \cdot d(s_i,e_i)+r_0} \]
Thus, $sl(\alpha_i)$ grows unboundedly as $i$ goes to infinity because the slope of the curves $\alpha_i^X$ goes to infinity in $X$. This is impossible as $\alpha_i$'s are quasi-geodesics in $X_\omega$. Therefore, it follows that the paths $\{\alpha^X_i\}$ also have bounded slopes.
\end{proof}
Thus the preimages all have bounded slope and by the path characterization, 
they all intersect the middle third of a sublinearly neighborhood. Which contradicts lemma~\ref{neighborhood}. Thus $\gamma_\omega$ is a sublinearly Morse set.
\end{proof}
\begin{remark}
The argument of Claim~\ref{39} makes it clear that we can only establish our main claim for sublinearly Morse boundary as opposed to for a specific $\kappa$-Morse boundary. 
\end{remark}
\subsection{Construction of the bijective map between sublinearly Morse boundaries} \label{Remark1}

First we observe that  a sublinear Morse geodesic ray can be mapped to a sublinearly Morse geodesic ray.
We first observe that one can always find a geodesic ray in a sublinear bound of the  set $\gamma_\omega$. This follows from Theorem 3.8 in \cite{JanaQing1}. 
Furthermore, the construction of the geodesic line is the following: consider a sequence of points on the set
$\gamma_\omega$ whose norm goes to infinity. The sequence $\{ x_i | d(x_i, \go_\omega) > i\}$ exists as $\gamma_\omega$ is an unbounded set. Consider the geodesic images $[\go, x_i]$ and take the limit. We conclude that since $\gamma_\omega$ is sublinearly Morse, by definition there exists a sublinear neighborhood such that every $[\go, x_i]$ ( a (q,0)-quasigeodesic segment with endpoints on $\gamma_\omega$). Hence the limit set $\gamma_0$ is in a sublinear neighborhood of $\gamma_\omega$ and is a geodesic ray we will refer to as a \emph{geodesic representative} of the class $[\gamma_\omega]$. 
\begin{definition}\label{maponray}
Define:
\[f^\star_\omega (\gamma) : = \gamma_0\]
It follows from \cite[Lemma 4.2]{QRT2} that $\gamma_0$ is also sublinearly Morse.
\end{definition}
First it needs to be verified that $f^\star_\omega$ is well-defined.

\begin{lemma} \label{welldefined}
Suppose $\alpha$ and $\beta$ are two geodesics in the same equivalence class $[\gamma]$ (with respect to sublinear tracking), and suppose there are two geodesic representatives $\alpha^\omega$ and $\beta^\omega$ in $X_\omega$ as results of Construction~\ref{Remark1}. Then
    $\alpha_{\omega}$ and $\beta_{\omega}$ are also in same equivalence class in $X_\omega$ i.e. $\frac{d_\omega\big(\alpha_{\omega}(t),\beta_{\omega}(t)\big)}{t} \rightarrow 0$, as $t \rightarrow \infty$.
\end{lemma}
\begin{proof}
   Let $\alpha^\omega$ and $\beta^\omega$ be the image of $\alpha$ and $\beta$ under a first passage percolation with the stated assumptions. Suppose $\alpha$ and $\beta$ are in the same class. Let $\gamma_t$ be a geodesic joining  $\alpha(t)$ and $\beta(t)$. Therefore by Proposition ~\ref{BT1} we have,  \[d_\omega(\alpha(t),\beta(t)) \leq |\gamma_t|_\omega \leq b\cdot d(\alpha(t),\beta(t))+r_0\] implies,\[\frac{d_\omega(\alpha(t),\beta(t))}{t}\leq b\cdot \frac{d(\alpha(t),\beta(t))}{t}+\frac{r_0}{t}\]
   Since $[\alpha]=[\beta]$, then $\frac{d_\omega(\alpha(t),\beta(t))}{t}\rightarrow0$, as $t \rightarrow\infty$. Hence, $\alpha^\omega \sim \beta^\omega$ in $X_\omega$. From the construction in line ~\ref{Remark1} we get $\alpha_{\omega}$ and $\beta_{\omega}$ refers to the geodesic representative of the class $[\alpha^\omega]$and $[\beta^\omega]$ in $X_\omega$. Since,   by transitivity we have $\alpha_{\omega}$ and $\beta_{\omega}$ are in same class in $X_\omega$.
\end{proof}

Now are ready produced the map constructed so far that is well-defined.
\begin{definition}[First passage percolation induces a map from $\ps X$ to $\ps X_\omega$.]\label{themap}
 From the Theorem ~\ref{37}, Construction$~\ref{Remark1}$, for almost every $\omega \in (\Omega, \mathbb{P})$ there is a  map $f_\omega^{\star}$ that takes sublinear geodesic rays to sublinear geodesic rays. Lemma~\ref{welldefined} shows that the maps preserves sublinear tracking of geodesic rays. Therefore we use the same notation to denote the associated map on a sublinear equivalence class  for every $\kappa$-Morse geodesic $\alpha$ in $X$ starting at $\mathfrak{o}$. Define:
 \[f_\omega^{\star}([\alpha]):=[\alpha^\omega].\]

\end{definition}

\begin{lemma} \label{oneone}
    The map $f_\omega^{\star}$  is one-to-one.
\end{lemma}
\begin{proof}
Suppose $f_\omega^\star([\alpha])=[\alpha_\omega]$ and  $f_\omega^\star([\beta])$=$[\beta_{\omega}]$ represents the same point in $\partial_sX_\omega$, where $\alpha,\beta$ are  geodesics in $X$, starting at $\go$. We aim to show that $[\alpha]=[\beta]$ in $\partial_sX$. Since $ [\alpha_\omega]=[\beta_\omega]$ then $\frac{d_\omega\big(\alpha^{\omega}_r,\beta^\omega_r\big)}{r} \rightarrow 0$ as $r\rightarrow \infty$. Suppose by way of contradiction, $\alpha$ and $\beta$ do not fellow travel each other sublinearly in $X$.
Now, consider $\{r_n\}_n$, a sequence of natural number such that for each $n\in \mathbb{N}$,  $\alpha^\omega_{r_n}:=\alpha^\omega(t_{r_n}), \text{and } \beta^\omega_{r_n}:=\beta^\omega(t_{r_n})$ be two points on $\alpha_\omega$ and $\beta_\omega$, where, $t_{r_n}$ be the first time where $||\alpha^\omega(t_{r_n})||_\omega=||\beta^\omega(t_{r_n})||_\omega=r_n.$ Let $s_{r_n},t_{r_n}$ be the nearest point projection of $\alpha^\omega_{r_n}$ and $\beta^\omega_{r_n}$ on $f_\omega(\alpha)$ and $f_\omega(\beta)$ respectively. Hence we have, $||s_{r_n}||_\omega=2||\beta^\omega_{r_n}||_\omega \leq 2r_n$ and similiarly, $||t_{r_n}||_\omega\leq 2r_n.$
By definition ~\ref{themap}, \[d_\omega(\alpha^\omega_{r_n},s_{r_n}) \leq m_1 \cdot \kappa(r_n) \text{  and }d_\omega(\beta^\omega_{r_n},t_{r_n}) \leq m_2 \cdot \kappa'(r_n).\]
Therefore, \[r_n-m_1\cdot \kappa(r_n) \leq ||s_{r_n}||_\omega \leq 2r_n  \text{ and  } r_n-m_2\cdot \kappa'(r_n) \leq ||t_{r_n}||_\omega \leq 2r_n .\]
If we see $s_{r_n}, t_{r_n}$ in $X,$ then by Proposition ~\ref{BT1} and ~\ref{BT2}, \begin{align*}
\frac{r_n-m_1\cdot \kappa(r_n)-r_0}{2b} &\leq ||s_{r_n}|| \leq \frac{2r_n+r_1}{c}, \text{ and  } \\
 \frac{r_n-m_2\cdot \kappa'(r_n)-r_0}{2b} &\leq ||s_{r_n}|| \leq \frac{2r_n+r_1}{c}.
\end{align*}
Therefore, we have $\frac{c}{4b}\leq \lim_{n\to\infty}\frac{||s_{r_n}||}{||t_{r_n}||} \leq \frac{4b}{c}$. Moreover, $||s_{r_n}||,||t_{r_n}|| \to \infty.$


Let $P_{r_{n}}$ be a geodesic in $X_\omega$ joining $s_{r_{n}}$ and $t_{r_{n}}$ in $X_\omega$. Now, $P_{r_{n}}^\omega$ be the preimage path of $P_{r_{n}}$ joining $s_{r_{n}}$ and $t_{r_{n}}$ in $X$.  Since $s_{r_{n}}$ and $t_{r_{n}}$ have normals of bounded ratio, there exists $C' >0$ such that 
\[d_X( \go,s_{r_{n}})  \leq C' \frac{s_{r_{n}}}{t_{r_{n}}}, \]
By Proposition~\ref{BT2}, the lengths of $P_{r_{n}}^\omega$ are bounded above sublinearly by $\Norm P_{r_{n}}^\omega$.
However, since $\alpha, \beta$ do not sublinearly fellow travel, there exists $C''$ and an infinite collection of radii $\{\tau_i\}$ such that 
\[ d(\beta(\tau_i), \alpha) \geq C'' \Norm {\beta(\tau_i)}. \]
Since $\Norm P_{r_{n}}^\omega$ occurs infinitely many times and so do $\{\tau_i\}$, then $\beta$ travels back and forth from linearly far away from $\alpha$ to sublinearly far away from $\alpha$. This is only possible if the $P_{r_{n}}^\omega$'s (thus accordingly the $s_{r_n}$) are spaced linearly apart. However, since $\alpha^\omega$ and $\beta^\omega$ sublinearly fellow travel each other, any time sampling will render the same result for the associated $P_{r_{n}}^\omega$'s. Thus, pick the initial sequence $\{r_n\}$ such that there are $n_2$ points in any interval $[0, n]$, and we achieve the desired contradiction. 
\end{proof} 

\begin{lemma} \label{onto}
    The map $f_\omega^{\star}$  is onto.
\end{lemma}
\begin{proof}
    Suppose the map is not onto. Therefore, there doesn't exist any $[a] \in \pka X$  such that $f_\omega([a])=[a^\omega]$ for some $\kappa$- Morse geodesic ray $a^\omega$ in $X_w$ starting at $\mathfrak{o}.$ That is, there does not exists any $\kappa$-Morse geodesic ray starting at $\mathfrak{o}$ whose image is in $[a^\omega] $ under $f_\omega$. Hence, for $\gamma^\omega \in [a^\omega]$ the pre image of $\gamma^\omega$ (say $\gamma$) under FPP map is not $\kappa$-Morse. Then there exists $(q, Q)$, $c_1 > 0$ and a family of $(q, Q)$-segments $\{\alpha_i \}$ with endpoints on $\gamma$ such that for each $i \in \NN$, there exists $\alpha_i$ whose starting point is of norm greater than $i$ and $\alpha_i$ is not contained in the $c_1$-linear neighborhood of $\gamma_\omega$. There is an infinite subsequence of {$\alpha_{i_k}$} with $d(s_{i_k},e_{i_k})\geq C\cdot d(s_{i_k},\mathfrak{o})$ for some constant $C$ where $s_{i_k},e_{i_k}$ are the endpoints of $\alpha_{i_k}$ on $\gamma$. By proposition ~\ref{BT2} for a.e. $\omega$  we have $d_\omega(s_{i_k},e_{i_k})\geq c \cdot d(s_{i_k},e_{i_k})-r_1$. Therefore,
    \[\frac{|\alpha_{i_k}^\omega|_\omega}{d_\omega(s_{i_k},e_{i_k})} \leq \frac{b \cdot |\alpha_{i_k}|+r_o}{c \cdot d(s_{i_k},e_{i_k})-r_0}\]
    Since  {$\alpha_{i_k}^\omega$} are quasi-geodesics as $i_k \to \infty$ we have $\frac{|\alpha_{i_k}|}{d(s_{i_k},e_{i_k})}$ is bounded. Hence, for large $i_k$, {$\alpha_{i_k}^\omega$} are paths with endpoints on $\gamma^\omega$ with bounded slope. Since $\gamma^\omega$ is $\kappa$-Morse so by proposition~\ref{prop3.3} $\alpha_{i_k}^\omega$ intersects a sublinear neighborhood of middle third of $\gamma^\omega[s_{i_k},e_{i_k}].$ Suppose $p_{i_k} \in \alpha_{i_k}$ be the point on intersection.  
   \begin{claim} \label{313}
     For a.e $\omega$ there exist $A=A(\omega), B=B(\omega)$ such that, \[\kappa(d_\omega(\mathfrak{o},p_{i_k})\geq d_\omega(p_{i_k},\gamma^\omega_{\frac 13[s_{i_k},e_{i_k}]}) \geq A \cdot d(\mathfrak{o},p)-B\]  
   \end{claim}
    \begin{proof}
    Since  $\alpha_i{_k} \cap \mathcal{N}_{\kappa'} (\gamma^\omega_{\frac 13[s_{i_k},e_{i_k}]}, m) \neq \emptyset$ then there exist $x$ on $\gamma^\omega$ such that, \[ d_\omega(p_{i_k}, \gamma^\omega)=d_\omega(p_{i_k}, x) \leq m \cdot \kappa(d_\omega(\mathfrak{o},p_{i_k}).\] Suppose $\sigma^\omega$ be the geodesic joining $p_{i_k}$ and $x$ in $X_\omega$ and under FPP map the preimage $\sigma$ be the path joining $p_{i_k}$,$x$ in $X$. Since $p_{i_k} \in \alpha_{i_k}$ then $p_{i_k}$ is outside the $c_1$ linear neighborhood. Therefore we get, 
    \[|\sigma|\geq d(p_{i_k},x)\geq d(p_{i_k},\gamma)\geq c_1 \cdot d(\mathfrak{o},p_{i_k}) \geq c_1 \cdot d(\mathfrak{o}, \sigma)\]
    Thus, by proposition ~\ref{BT2} for a.e $\omega$ there exist $r_1=r_1(\omega)$ such that, \[ d_\omega(p_{i_k},x)\geq c \cdot d(p_{i_k},x)-r_1.\] Hence, the claim follows.
    \end{proof}
    By proposition ~\ref{BT1}, ~\ref{313} and monotonicity of $\kappa$ we have, \[\kappa\big(b \cdot d(\mathfrak{o},p_{i_k})+r_o\big) \geq A \cdot d(\mathfrak{o},p_{i_k})-B.\]
    Which is a contradiction.
    \end{proof}
    


\begin{lemma} \label{kappaboundinFPP}
Let $\gamma \in X$ be a sublinearly Morse geodesic ray. Then in $X_\omega$, $\gamma_\omega$ and the geodesic ray $\gamma_0$ are in sublinear neighborhoods of each other.
\end{lemma}
\begin{proof}
Consider the limit geodesic of $[\go, c_i]$ and call that $\gamma_1$,  By Definition~\ref{maponray},  $\gamma_1$ is a sublienarly Morse geodesic ray. By the construction of  $\gamma_1$, $\gamma_1$ is in a sublinear neighborhood of $\gamma_\omega$. Given $\ep>0$, there exists a subsequence of $c_{ij}$ that are within $\ep$-neighborhood of $\gamma_1$, therefore also in a sublinear neighbohrood of $\gamma_0$, contradicting the choice of the $c_i$.

Suppose $\gamma$ is a $\kappa$-Morse geodesic ray with morse gauge $M$ and $\omega(\alpha)$ is a sublinear Morse set for some sublinear function $\rho$. Now, by Construction~\ref{Remark1}, $\gamma_0$ is the geodesic representative. Let 
\begin{equation}\label{linearnbd} 
  \end{equation}
  denote a $c_1$-\emph{linear neighborhood} of a set $\beta$.
Suppose $\{c_i\}_{i\geq1}$ are the points on $\gamma$ (also on $\omega(\gamma))$ for which Equation~\ref{linearnbd} satifies and $d_\omega(\go,c_{i})\geq i$. Now if we take $[\go,c_i]_\omega$ then the limit set namely $\gamma_1$ is in a $\rho$-neighbourhood of $\omega(\gamma).$ Suppose by the construction the geodesic representative of preimage of $\gamma_1$ under $f_\omega^\star$ is $\alpha$. Then by construction, $\{c_i\}$'s are on $\alpha$. For each $i$, the segment $\alpha|_{[c_i,c_{i+1}]}$ is a geodesic therefore, $\alpha|_{[c_i,c_{i+1}]} \subset \calN_\kappa(\gamma, m), \text{ where $m=M(1,0)$}$. Therefore, $\alpha \subset  \calN_\kappa(\gamma, m)$. This implies $\alpha \sim \gamma$. By the injectivity of $ f_\omega^\star$, $\gamma_0 \sim \gamma_1$ in $X_\omega$. By \cite[Lemma 3.4]{QRT2}, $\gamma_0 \subset \calN_\rho(\gamma_1, m_{\gamma_1}(1,0))$ and $\gamma_1 \subset\calN_\rho(\gamma_0, 2m_{\gamma_1}(1,0))$.
\newline Suppose $x_{ij}$ are the points on $\gamma_1$ such that $d_\omega(x_{ij},c_{ij}) <\epsilon$. Also observe that, $d_\omega(\go,c_{ij})\leq d_\omega(\go,c_{ij})+ \epsilon$.
Hence,
\[c\cdot d_\omega(\go,c_{ij}) \leq d(c_{ij}, \gamma_0) \leq d_\omega(c_{ij},x_{ij})+d(x_{ij},\gamma_0)\]
Which implies, $c\cdot d_\omega(\go,c_{ij}) \leq \epsilon+2m \cdot \rho(d_\omega(\go,x_{ij}))\leq  \epsilon+2m \cdot \rho(d_\omega(\go,c_{ij})+ \epsilon)$. Which is a contradiction to the fact that $d_\omega(c_{ij},\go)\rightarrow \infty$.

\end{proof}

\begin{theorem}\label{invarianttopology}
Let $f_\omega: X \rightarrow X_\omega$ be a first passage percolation where $\nu({0})=0$.  Then for a.e $\omega $ $f_\omega$ induces a homeomorphism 
$f^\star_\omega \from \partial_s X \to \partial_s X_\omega$.
\end{theorem}

\begin{proof}

By Lemma ~\ref{oneone} and Lemma ~\ref{onto}, $f_\omega^\star$ defined in Definition ~\ref{maponray} gives a bijection between $\partial_sX$ 
and $\partial_s Y$. To prove the homeomorphism, we first show that $f_\omega^{\star^{-1}}$ is continuous, and secondly, we will show $f_\omega^\star$ is continuous. 
\newline To prove the first claim, we will use the classical fact that  $f_\omega^{\star}$ is an open map
if and only if the following holds:
For any $\bfb_X \in \ps X$ and $\bfb_Y : =f_\omega^\star(\bfb_X)$ and any $ \rr> 0$, there exists $\rr'$  such that if $\bfa_Y\in \calU_\kappa (\bfb_Y,\rr')$ then there exists $\bfa_X \in \calU_\kappa (\bfb_X,\rr)$ such that \[ (f_\omega^\star)^{-1}(\bfa_Y)=\bfa_X.\]

Assume without loss of generality that $\bfb_X$ is $\kappa$-Morse for some $\kappa$. Let $\bfb_Y = f_\omega^\star(\bfb_X)$. If $b_X$ is the geodesic representative $\bfb_Y$ then $f_\omega(b_X)$ is a $\kappa'$-Morse set for some $\kappa'$ and $b_Y$ is the geodesic representative of $\bfb_Y \in \partial_sX$ from the definition. Therefore by construction $b_Y \subset \calN_{\kappa'}(f_\omega(b_X),\mm_{f_\omega(b_X)})$  
similarly if $\bfa_X\in \partial_sX $ then $a_Y \subset \calN_{\kappa'}(f_\omega(a_X),\mm_{f_\omega(a_X)})$ where $a_X$ is a $\rho$-Morse geodesic representative in $\bfa_X$ whereas the image $f_\omega(a_X)$ is $\rho'$-Morse. By Lemma ~\ref{kappaboundinFPP}, there is a constant $\nn_1$ depending on $\kappa, \kappa'$ and $\mm_{b_Y}$ such that 
\[
f_\omega( a_X) \subset \calN_{\rho'}(a_Y, \nn_1). 
\]
Next, we let 
\[
\nn =  C(r)^{-1}[\nn_1 +2n+4\mm_{f_\omega(b_X)}(1,0)+r_0(\omega)] \cdot (b'+r_0'(\omega)) 
\]
Recall, $r_0(\omega)$ and $b$ are notations defined in Proposition ~\ref{BT1} and Proposition~\ref{BT2} and  $r_0'(\omega)$, $b'$ are constants depending on $b, r_0(\omega).$
and let $\sR = \sR(b_X, \rr, \nn, \kappa')$ as in Definition of sublinear Morse 
choose $\rr'$ such that:
\begin{enumerate}
\item $\rr' \geq 2bR + r_0(\omega) \cdot \kappa(r)$, and
\item $n$ is small  compare to $\rr'$.
\end{enumerate}
\begin{figure}
\tikzset{every picture/.style={line width=0.75pt}} 

\begin{tikzpicture}[x=0.75pt,y=0.75pt,yscale=-1,xscale=1]

\draw    (36,166.4) -- (176.91,167.33) ;
\draw [color={rgb, 255:red, 208; green, 2; blue, 27 }  ,draw opacity=1 ][fill={rgb, 255:red, 208; green, 2; blue, 27 }  ,fill opacity=1 ]   (36,166.4) -- (148.55,76.05) ;
\draw   (256.46,149.96) .. controls (258.96,151.67) and (261.35,153.3) .. (264.35,153.87) .. controls (267.36,154.44) and (270.15,153.79) .. (273.07,153.1) .. controls (275.99,152.42) and (278.78,151.77) .. (281.78,152.34) .. controls (284.79,152.91) and (287.18,154.54) .. (289.68,156.25) .. controls (292.17,157.96) and (294.57,159.6) .. (297.57,160.17) .. controls (300.58,160.74) and (303.37,160.09) .. (306.29,159.4) .. controls (309.2,158.71) and (311.99,158.06) .. (315,158.63) .. controls (318.01,159.2) and (320.4,160.83) .. (322.89,162.55) .. controls (325.39,164.26) and (327.78,165.89) .. (330.79,166.46) .. controls (333.79,167.03) and (336.58,166.38) .. (339.5,165.69) .. controls (342.42,165.01) and (345.21,164.35) .. (348.22,164.92) .. controls (351.22,165.49) and (353.61,167.13) .. (356.11,168.84) .. controls (358.61,170.55) and (361,172.19) .. (364,172.76) .. controls (367.01,173.33) and (369.8,172.67) .. (372.72,171.99) .. controls (375.64,171.3) and (378.43,170.65) .. (381.43,171.22) .. controls (384.44,171.79) and (386.83,173.42) .. (389.33,175.13) .. controls (391.83,176.85) and (394.22,178.48) .. (397.22,179.05) .. controls (400.23,179.62) and (403.02,178.97) .. (405.94,178.28) .. controls (408.85,177.6) and (411.65,176.94) .. (414.65,177.51) .. controls (417.66,178.08) and (420.05,179.72) .. (422.54,181.43) .. controls (425.04,183.14) and (427.43,184.78) .. (430.44,185.34) .. controls (433.44,185.91) and (436.23,185.26) .. (439.15,184.58) .. controls (441.92,183.93) and (444.56,183.31) .. (447.39,183.73) ;
\draw  [color={rgb, 255:red, 208; green, 2; blue, 27 }  ,draw opacity=1 ][dash pattern={on 4.5pt off 4.5pt}] (256.88,150.08) .. controls (259.73,149.13) and (262.45,148.22) .. (264.76,146.15) .. controls (267.07,144.07) and (268.33,141.39) .. (269.64,138.6) .. controls (270.96,135.8) and (272.21,133.12) .. (274.53,131.05) .. controls (276.84,128.97) and (279.56,128.06) .. (282.41,127.11) .. controls (285.27,126.16) and (287.99,125.26) .. (290.3,123.18) .. controls (292.61,121.1) and (293.87,118.43) .. (295.18,115.63) .. controls (296.49,112.83) and (297.75,110.16) .. (300.06,108.08) .. controls (302.37,106) and (305.1,105.09) .. (307.95,104.14) .. controls (310.8,103.2) and (313.53,102.29) .. (315.84,100.21) .. controls (318.15,98.13) and (319.41,95.46) .. (320.72,92.66) .. controls (322.03,89.86) and (323.29,87.19) .. (325.6,85.11) .. controls (327.91,83.03) and (330.63,82.12) .. (333.49,81.17) .. controls (336.34,80.23) and (339.06,79.32) .. (341.37,77.24) .. controls (343.68,75.16) and (344.94,72.49) .. (346.25,69.69) .. controls (347.57,66.89) and (348.82,64.22) .. (351.14,62.14) .. controls (353.45,60.06) and (356.17,59.15) .. (359.02,58.2) .. controls (361.88,57.26) and (364.6,56.35) .. (366.91,54.27) .. controls (369.22,52.19) and (370.48,49.52) .. (371.79,46.72) .. controls (373.1,43.92) and (374.36,41.25) .. (376.67,39.17) .. controls (378.98,37.09) and (381.71,36.18) .. (384.56,35.23) .. controls (387.41,34.29) and (390.14,33.38) .. (392.45,31.3) .. controls (394.76,29.22) and (396.02,26.55) .. (397.33,23.75) .. controls (398.57,21.1) and (399.76,18.56) .. (401.85,16.54) ;
\draw [color={rgb, 255:red, 208; green, 2; blue, 27 }  ,draw opacity=1 ][fill={rgb, 255:red, 208; green, 2; blue, 27 }  ,fill opacity=1 ]   (257.94,150.48) -- (432.61,51.23) ;
\draw  [color={rgb, 255:red, 144; green, 19; blue, 254 }  ,draw opacity=1 ][dash pattern={on 0.84pt off 2.51pt}] (372.74,194.02) .. controls (365.5,134.42) and (337.11,92.71) .. (287.56,68.87) ;
\draw  [line width=3] [line join = round][line cap = round] (90.92,122.39) .. controls (90.92,122.39) and (90.92,122.39) .. (90.92,122.39) ;
\draw  [line width=3] [line join = round][line cap = round] (116.58,167.33) .. controls (116.58,167.18) and (116.28,167.33) .. (116.13,167.33) ;
\draw  [line width=3] [line join = round][line cap = round] (357.88,149.08) .. controls (357.88,149.08) and (357.88,149.08) .. (357.88,149.08) ;
\draw    (203.02,112.09) -- (234.33,112.53) ;
\draw [shift={(236.33,112.56)}, rotate = 180.81] [color={rgb, 255:red, 0; green, 0; blue, 0 }  ][line width=0.75]    (10.93,-3.29) .. controls (6.95,-1.4) and (3.31,-0.3) .. (0,0) .. controls (3.31,0.3) and (6.95,1.4) .. (10.93,3.29)   ;
\draw  [color={rgb, 255:red, 144; green, 19; blue, 254 }  ,draw opacity=1 ][dash pattern={on 0.84pt off 2.51pt}] (126.49,209) .. controls (119.25,149.41) and (90.86,107.69) .. (41.31,83.85) ;
\draw    (257.94,150.48) -- (424.73,148.27) ;

\draw (153.97,60.12) node [anchor=north west][inner sep=0.75pt]  [font=\scriptsize]  {$\alpha _{X}$};
\draw (357.67,11.61) node [anchor=north west][inner sep=0.75pt]  [font=\footnotesize,color={rgb, 255:red, 208; green, 2; blue, 27 }  ,opacity=1 ]  {$f_{\omega }( \alpha _{X})$};
\draw (436.04,37.93) node [anchor=north west][inner sep=0.75pt]  [font=\footnotesize,color={rgb, 255:red, 208; green, 2; blue, 27 }  ,opacity=1 ]  {$f_{\omega }^{\star }( \alpha _{X})$};
\draw (178.36,160.27) node [anchor=north west][inner sep=0.75pt]  [font=\footnotesize]  {$b_{X}$};
\draw (424.25,189.85) node [anchor=north west][inner sep=0.75pt]  [font=\footnotesize]  {$f_{\omega }( b_{X})$};
\draw (427.81,136.36) node [anchor=north west][inner sep=0.75pt]  [font=\footnotesize]  {$f_{\omega }^{\star }( b_{X})$};
\draw (76.89,111.1) node [anchor=north west][inner sep=0.75pt]    {$x$};
\draw (105.67,166.88) node [anchor=north west][inner sep=0.75pt]    {$p$};
\draw (351.54,151.44) node [anchor=north west][inner sep=0.75pt]  [font=\scriptsize]  {$y'$};
\draw (209.06,79.86) node [anchor=north west][inner sep=0.75pt]    {$f_{\omega }$};
\draw (319.45,96.66) node [anchor=north west][inner sep=0.75pt]  [font=\scriptsize]  {$n_{1}$};
\draw (106.26,95.3) node [anchor=north west][inner sep=0.75pt]  [font=\scriptsize,color={rgb, 255:red, 208; green, 2; blue, 27 }  ,opacity=1 ,rotate=-321.74,xslant=0.04]  {$\rho$ --Morse};
\draw (370.14,60.94) node [anchor=north west][inner sep=0.75pt]  [font=\scriptsize,color={rgb, 255:red, 208; green, 2; blue, 27 }  ,opacity=1 ,rotate=-321.74,xslant=0.04]  {$\rho '$--Morse};
\draw (120.23,169.54) node [anchor=north west][inner sep=0.75pt]  [font=\scriptsize,color={rgb, 255:red, 0; green, 0; blue, 0 }  ,opacity=1 ,rotate=-0.49,xslant=0.04]  {$\kappa$ --Morse};
\draw (400.49,155.66) node [anchor=north west][inner sep=0.75pt]  [font=\scriptsize,color={rgb, 255:red, 0; green, 0; blue, 0 }  ,opacity=1 ,rotate=-14.2,xslant=0.04]  {$\kappa '$--Morse};
\draw (286.89,54.87) node [anchor=north west][inner sep=0.75pt]  [color={rgb, 255:red, 144; green, 19; blue, 254 }  ,opacity=1 ]  {$r'$};
\draw (42.62,89.26) node [anchor=north west][inner sep=0.75pt]  [color={rgb, 255:red, 144; green, 19; blue, 254 }  ,opacity=1 ]  {$R$};
\draw (160.51,15.06) node [anchor=north west][inner sep=0.75pt]  [color={rgb, 255:red, 65; green, 117; blue, 5 }  ,opacity=1 ]  {$X$};
\draw (304.75,10.43) node [anchor=north west][inner sep=0.75pt]  [color={rgb, 255:red, 65; green, 117; blue, 5 }  ,opacity=1 ]  {$X_{\omega }$};
\draw (209.1,125.03) node [anchor=north west][inner sep=0.75pt]  [font=\footnotesize]  {$f_{\omega }^{\star }$};

\end{tikzpicture}

\caption{ $f_\omega^{\star}$ is an open map.}
\end{figure}
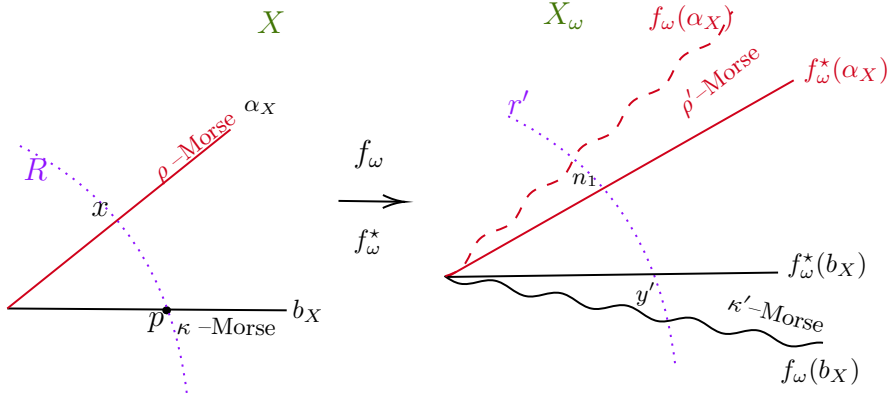

Now we are ready to consider any given  $\alpha \in \bfa_X$ be a $(\qq, \sQ)$--quasi-geodesic in $X$, where
\begin{enumerate}
\item $\qq, \sQ$ is such that  $\mm_{b_X}(\qq, \sQ)$ is small compared to $\rr$ with respect to $\kappa$; and
\item $f_\omega(\alpha)$ is a ray that is associated with an element in $\calU_{\kappa'}(\bfb_Y, \rr')$. 
\end{enumerate}
By the choice of $\rr'$, 
$n$ is small compared to $\rr'$. Hence, 
\[
(f_\omega^\star)(\alpha_X)=\alpha_Y|_{\rr'} \subset \calN_{\kappa'}\big(b_Y, n\big)
\]

Pick a point $x \in \alpha_X|_{\sR}$ such that $||x||=R$ in $X$. Lets consider $d(x,b_X)$ then  $d(x,b_X)=d(x,p)$ where $p \in \Pi_{b_X}(x)$. Then if $[x,p]_\omega$ is the $\omega$-geodesic joining $x$ and $p$. Then the pre image path namely, $\beta$ satisfy, $d(\go,\beta) \leq ||x||=R \leq R \cdot \ell(\beta)$. Therefore by Proposition ~\ref{BT1}, we have
\[d(x,b_X) \leq C(R)^{-1}(d_\omega(x,b_X)  + r_o)\]


\begin{align*} 
d_X(x, b_X) 
&\leq C(R)^{-1}(d_\omega(x,f_\omega(b_X))  + r_o)  \\
& \leq C(R)^{-1}[d_\omega(x,\yy)+d_\omega(\yy,\yy')+d_\omega(\yy',f_\omega(b_X)] \text{\hspace{0.3cm} (where, } \yy \in \Pi_{\alpha_Y}(x), \yy' \in \Pi_{b_Y}(x))\\
& \leq C(R)^{-1}[\nn_1 \cdot \rho'(||x||_\omega)+n \cdot \kappa'(||\yy||_\omega)+\mm_{f_\omega(b_X)}(1,0) \cdot \kappa'(||\yy'||_\omega)+r_0(\omega)]\\
& \leq   C(R)^{-1}[\nn_1  +2n+4\mm_{f_\omega(b_X)}(1,0)+r_0(\omega)] \cdot \calA(||x||_\omega)
\end{align*} 
where, $\calA=\max\{\rho',\kappa'\}$.
We also have 
\[
\calA(||x||_\omega) \leq b'\cdot \calA(||x||)+r_o'  \leq (b'+r_0'(\omega)) \calA(||x||)
\]
Note that by Proposition~\ref {extension},  $C(\param)$ is a decreasing function, hence, combining the preceding inequalities, we have,

\begin{align*} 
d_X(x, b_X) & \leq C(R)^{-1}[\nn_1  +2n+4\mm_{f_\omega(b_X)}(1,0)+r_0(\omega)] \cdot \calA(||x||_\omega)\\
& \leq C(R)^{-1}[\nn_1  +2n+4\mm_{f_\omega(b_X)}(1,0)+r_0(\omega)] \cdot (b'+r_0'(\omega) \calA(||x||)\\
& \leq C(r)^{-1}[\nn_1 +2n+4\mm_{f_\omega(b_X)}(1,0)+r_0(\omega)] \cdot (b'+r_0'(\omega) \calA(||x||)\\
& \leq \nn \cdot \calA(||x||)
\end{align*}

Now, Definition ~\ref{Def:Morse} implies that 
\[
\alpha_X|_\rr \subset \calN_\kappa(b_X, \mm_{b_X}(1,0)). 
\]
Therefore, $\bfa_X \in \calU_{\kappa}(\bfb_X, \rr)$ and
\[
(f_\omega^\star)^{-1} \big(\calU_{s}(\bfb_Y, \rr')\big) \subset \calU_{s}(\bfb_X, \rr). 
\]

For the other direction, it is equivalent to saying, the $(f_\omega^\star)^{-1}$ is open. Fix a $r >0$, then we aim to show that, for any  $\calU_\kappa(\bfb_Y,r)$ the pre image $(f_\omega^\star)^{-1}(\calU_\kappa(\bfb_Y,r))$ is open in $\partial_sX$. That is, for every $\bfa_X\in (f_\omega^\star)^{-1}(\calU_\kappa(\bfb_Y,r)) $ there exists a $r'$, such that $\calU_s(\bfa_X,r') \subset (f_\omega^\star)^{-1}(\calU_\kappa(\bfb_Y,r)).$ Without loss of generality, suppose $f_\omega^\star(\bfa_X)=\bfa_Y.$ \newline
Let $R(b_Y,r,\nn, \rho')$ then choose $r'$ such that:
\begin{enumerate}
    
    \item $r' \geq \frac{R+r_1}{c}$, where $r_1, c$ are the constant from Proposition ~\ref{BT2}.
    \item $\nn'$ is small compared to $r'$.
    \end{enumerate}
Suppose, $\bfb_X \in \calU_s(\bfa_X,r')$ and $b_X$ be a geodesic representative of $\bfb_X \in \partial_sX$ and $f_\omega^\star(\bfb_X)=\bfb_Y.$ 
Suppose $a_Y$ and $b_Y$ are geodesic representatives of $\bfa_Y$ and $\bfb_Y$ in $\partial_sX_{\omega}$ respectively.  

Suppose, $y$ be a point on $a_Y$ with $||y||_\omega=R$, and let $y' \in \Pi_y(\omega(a_X))$. Now, $y''$ be the nearest point projection of $y'$ on $a_X$ in $X$. Then,
\begin{align*}
  d_\omega(y, b_Y) & \leq d_\omega(y,y')+d_\omega(y',y'')+d_\omega(y'',b_Y)\\
  &\leq m_{b_Y}(1,0) \cdot \kappa'(R)+ 2b\cdot m_{a_X}(1,0) \cdot \rho(||y'||)+ r_0(\omega)+ n_1 \cdot \rho'(||y''||_\omega)  
\end{align*}
Similarly like the previous proof there exists $\nn'$ and a sublinear function $\eta$, we have \[d_\omega(y, b_Y) \leq \nn' \cdot \eta(R)\].
\[
a_Y|_\rr \subset \calN_\kappa(b_Y, \mm_{b_Y}(1,0)). 
\]
Hence, $f_\omega^*$ is continuous. 
\end{proof} 

Lastly, we would like to quickly observe that since sublinearly Morse boundaries are visibility spaces \cite[Theorem C]{DZ}, for any two directions $\alpha, \beta \in \pka X$, there exists a bi-infinite geodesic line in $X_\omega$ that converges to them. This is also an observation analogous to \cite[Theorem 1.13]{BM2}.
\begin{corollary}
    Let $X$ be an infinite graph with uniformly bounded degree. For almost every $\omega$, consider any two distinct element in $\pka X$,  there exists a bi-infinite geodesic line in $X_{\omega}$  that converges to these two directions in $\pka X_\omega$.
\end{corollary}
\begin{proof}

Fix distinct $\xi_1,\xi_2\in \partial_s X$. For almost every $\omega$, the map $f_\omega^\star:\partial_s X\to \partial_s X_\omega$ is a homeomorphism, hence
$f_\omega^\star(\xi_1)\neq f_\omega^\star(\xi_2)$. Choose $\kappa$-Morse
quasi-geodesic rays $\alpha_1^\omega,\alpha_2^\omega$ in $X_\omega$, both based
at $\go$, representing $f_\omega^\star(\xi_1)$ and $f_\omega^\star(\xi_2)$,
respectively. Since these boundary points are distinct, the rays
$\alpha_1^\omega$ and $\alpha_2^\omega$ do not $\kappa$-fellow travel. By the
Theorem 2.13 in \cite{DZ}, there exists a bi-infinite geodesic line
$\beta_\omega:\mathbb{R}\to X_\omega$ whose two half-rays are $\kappa$-equivalent
to $\alpha_1^\omega$ and $\alpha_2^\omega$. Hence the two ends of $\beta_\omega$
represent $f_\omega^\star(\xi_1)$ and $f_\omega^\star(\xi_2)$.

\end{proof}

\bibliographystyle{alpha}

\end{document}